\documentclass[12pt]{article}
\usepackage{amsmath,amssymb,amsthm}
\usepackage{hyperref}
\usepackage{tikz}

\def\picAA
{\begin{tikzpicture}[thick]
\draw[shift={( 0.0,0)}] (0,0.0) to[out=0,in=90] (0.75,-0.25) to[out=90,in=180] (1.5,0.0);
\draw[shift={( 0.0,0)}] (0,0.5) to[out=0,in=180] (1.5,0.5);
\draw[shift={( 0.0,0)}] (0,1.0) to[out=0,in=180] (1.5,1.0);
\draw[shift={( 0.0,0)}] (0,1.5) to[out=0,in=180] (1.5,1.5);
\draw[shift={( 0.0,0)}, color=black!30!white] (0,-0.25)--(1.5,-0.25)--(1.5,1.75)--(0,1.75)--cycle;
\node[shift={( 0.0,0)}] at (-0.25,0.0) {\small \emph1};
\node[shift={( 0.0,0)}] at (-0.25,0.5) {\small \emph2};
\node[shift={( 0.0,0)}] at (-0.25,1.0) {\small \emph3};
\node[shift={( 0.0,0)}] at (-0.25,1.5) {\small \emph4};
\node[shift={( 0.0,0)}] at ( 1.75,0.0) {\small \emph1};
\node[shift={( 0.0,0)}] at ( 1.75,0.5) {\small \emph2};
\node[shift={( 0.0,0)}] at ( 1.75,1.0) {\small \emph3};
\node[shift={( 0.0,0)}] at ( 1.75,1.5) {\small \emph4};

\draw[shift={( 3.0,0)}] (0,0.0) to[out=0,in=90] (0.75,-0.25) to[out=90,in=180] (1.5,0.5);
\draw[shift={( 3.0,0)}] (0,0.5) to[out=0,in=180] (1.5,0.0);
\draw[shift={( 3.0,0)}] (0,1.0) to[out=0,in=180] (1.5,1.0);
\draw[shift={( 3.0,0)}] (0,1.5) to[out=0,in=180] (1.5,1.5);
\draw[shift={( 3.0,0)}, color=black!30!white] (0,-0.25)--(1.5,-0.25)--(1.5,1.75)--(0,1.75)--cycle;
\node[shift={( 3.0,0)}] at (-0.25,0.0) {\small \emph1};
\node[shift={( 3.0,0)}] at (-0.25,0.5) {\small \emph2};
\node[shift={( 3.0,0)}] at (-0.25,1.0) {\small \emph3};
\node[shift={( 3.0,0)}] at (-0.25,1.5) {\small \emph4};
\node[shift={( 3.0,0)}] at ( 1.75,0.0) {\small \emph1};
\node[shift={( 3.0,0)}] at ( 1.75,0.5) {\small \emph2};
\node[shift={( 3.0,0)}] at ( 1.75,1.0) {\small \emph3};
\node[shift={( 3.0,0)}] at ( 1.75,1.5) {\small \emph4};

\draw[shift={( 6.0,0)}] (0,0.0) to[out=0,in=90] (0.75,-0.25) to[out=90,in=180] (1.5,1.0);
\draw[shift={( 6.0,0)}] (0,0.5) to[out=0,in=180] (1.5,0.0);
\draw[shift={( 6.0,0)}] (0,1.0) to[out=0,in=180] (1.5,0.5);
\draw[shift={( 6.0,0)}] (0,1.5) to[out=0,in=180] (1.5,1.5);
\draw[shift={( 6.0,0)}, color=black!30!white] (0,-0.25)--(1.5,-0.25)--(1.5,1.75)--(0,1.75)--cycle;
\node[shift={( 6.0,0)}] at (-0.25,0.0) {\small \emph1};
\node[shift={( 6.0,0)}] at (-0.25,0.5) {\small \emph2};
\node[shift={( 6.0,0)}] at (-0.25,1.0) {\small \emph3};
\node[shift={( 6.0,0)}] at (-0.25,1.5) {\small \emph4};
\node[shift={( 6.0,0)}] at ( 1.75,0.0) {\small \emph1};
\node[shift={( 6.0,0)}] at ( 1.75,0.5) {\small \emph2};
\node[shift={( 6.0,0)}] at ( 1.75,1.0) {\small \emph3};
\node[shift={( 6.0,0)}] at ( 1.75,1.5) {\small \emph4};

\draw[shift={( 9.0,0)}] (0,0.0) to[out=0,in=90] (0.75,-0.25) to[out=90,in=180] (1.5,1.5);
\draw[shift={( 9.0,0)}] (0,0.5) to[out=0,in=180] (1.5,0.0);
\draw[shift={( 9.0,0)}] (0,1.0) to[out=0,in=180] (1.5,0.5);
\draw[shift={( 9.0,0)}] (0,1.5) to[out=0,in=180] (1.5,1.0);
\draw[shift={( 9.0,0)}, color=black!30!white] (0,-0.25)--(1.5,-0.25)--(1.5,1.75)--(0,1.75)--cycle;
\node[shift={( 9.0,0)}] at (-0.25,0.0) {\small \emph1};
\node[shift={( 9.0,0)}] at (-0.25,0.5) {\small \emph2};
\node[shift={( 9.0,0)}] at (-0.25,1.0) {\small \emph3};
\node[shift={( 9.0,0)}] at (-0.25,1.5) {\small \emph4};
\node[shift={( 9.0,0)}] at ( 1.75,0.0) {\small \emph1};
\node[shift={( 9.0,0)}] at ( 1.75,0.5) {\small \emph2};
\node[shift={( 9.0,0)}] at ( 1.75,1.0) {\small \emph3};
\node[shift={( 9.0,0)}] at ( 1.75,1.5) {\small \emph4};

\node[shift={(0,0)}] at (0.75, 2.05) {$C_1$};
\node[shift={(3,0)}] at (0.75, 2.05) {$C_2$};
\node[shift={(6,0)}] at (0.75, 2.05) {$C_3$};
\node[shift={(9,0)}] at (0.75, 2.05) {$C_4$};
\end{tikzpicture}}

\def\picAC{\begin{tikzpicture}[thick]
\draw[shift={( 0.0,0)}] (0,0.0) to[out=0,in=90] (0.75,-0.25) to[out=90,in=180] (1.5,1.0);
\draw[shift={( 0.0,0)}] (0,0.5) to[out=0,in=180] (1.5,0.0);
\draw[shift={( 0.0,0)}] (0,1.0) to[out=0,in=180] (1.5,0.5);
\draw[shift={( 0.0,0)}] (0,1.5) to[out=0,in=180] (1.5,1.5);
\draw[shift={( 0.0,0)}, color=black!30!white] (0,-0.25)--(1.5,-0.25)--(1.5,1.75)--(0,1.75)--cycle;
\draw[shift={( 1.5,0)}] (0,0.0) to[out=0,in=90] (0.75,-0.25) to[out=90,in=180] (1.5,1.0);
\draw[shift={( 1.5,0)}] (0,0.5) to[out=0,in=180] (1.5,0.0);
\draw[shift={( 1.5,0)}] (0,1.0) to[out=0,in=180] (1.5,0.5);
\draw[shift={( 1.5,0)}] (0,1.5) to[out=0,in=180] (1.5,1.5);
\draw[shift={( 1.5,0)}, color=black!30!white] (0,-0.25)--(1.5,-0.25)--(1.5,1.75)--(0,1.75)--cycle;
\draw[shift={( 3.0,0)}] (0,0.0) to[out=0,in=90] (0.75,-0.25) to[out=90,in=180] (1.5,0.5);
\draw[shift={( 3.0,0)}] (0,0.5) to[out=0,in=180] (1.5,0.0);
\draw[shift={( 3.0,0)}] (0,1.0) to[out=0,in=180] (1.5,1.0);
\draw[shift={( 3.0,0)}] (0,1.5) to[out=0,in=180] (1.5,1.5);
\draw[shift={( 3.0,0)}, color=black!30!white] (0,-0.25)--(1.5,-0.25)--(1.5,1.75)--(0,1.75)--cycle;
\draw[shift={( 4.5,0)}] (0,0.0) to[out=0,in=90] (0.75,-0.25) to[out=90,in=180] (1.5,1.5);
\draw[shift={( 4.5,0)}] (0,0.5) to[out=0,in=180] (1.5,0.0);
\draw[shift={( 4.5,0)}] (0,1.0) to[out=0,in=180] (1.5,0.5);
\draw[shift={( 4.5,0)}] (0,1.5) to[out=0,in=180] (1.5,1.0);
\draw[shift={( 4.5,0)}, color=black!30!white] (0,-0.25)--(1.5,-0.25)--(1.5,1.75)--(0,1.75)--cycle;
\draw[shift={( 6.0,0)}] (0,0.0) to[out=0,in=90] (0.75,-0.25) to[out=90,in=180] (1.5,1.0);
\draw[shift={( 6.0,0)}] (0,0.5) to[out=0,in=180] (1.5,0.0);
\draw[shift={( 6.0,0)}] (0,1.0) to[out=0,in=180] (1.5,0.5);
\draw[shift={( 6.0,0)}] (0,1.5) to[out=0,in=180] (1.5,1.5);
\draw[shift={( 6.0,0)}, color=black!30!white] (0,-0.25)--(1.5,-0.25)--(1.5,1.75)--(0,1.75)--cycle;
\draw[shift={( 7.5,0)}] (0,0.0) to[out=0,in=90] (0.75,-0.25) to[out=90,in=180] (1.5,1.5);
\draw[shift={( 7.5,0)}] (0,0.5) to[out=0,in=180] (1.5,0.0);
\draw[shift={( 7.5,0)}] (0,1.0) to[out=0,in=180] (1.5,0.5);
\draw[shift={( 7.5,0)}] (0,1.5) to[out=0,in=180] (1.5,1.0);
\draw[shift={( 7.5,0)}, color=black!30!white] (0,-0.25)--(1.5,-0.25)--(1.5,1.75)--(0,1.75)--cycle;
\draw[shift={( 9.0,0)}] (0,0.0) to[out=0,in=90] (0.75,-0.25) to[out=90,in=180] (1.5,1.0);
\draw[shift={( 9.0,0)}] (0,0.5) to[out=0,in=180] (1.5,0.0);
\draw[shift={( 9.0,0)}] (0,1.0) to[out=0,in=180] (1.5,0.5);
\draw[shift={( 9.0,0)}] (0,1.5) to[out=0,in=180] (1.5,1.5);
\draw[shift={( 9.0,0)}, color=black!30!white] (0,-0.25)--(1.5,-0.25)--(1.5,1.75)--(0,1.75)--cycle;
\draw[shift={(10.5,0)}] (0,0.0) to[out=0,in=90] (0.75,-0.25) to[out=90,in=180] (1.5,0.5);
\draw[shift={(10.5,0)}] (0,0.5) to[out=0,in=180] (1.5,0.0);
\draw[shift={(10.5,0)}] (0,1.0) to[out=0,in=180] (1.5,1.0);
\draw[shift={(10.5,0)}] (0,1.5) to[out=0,in=180] (1.5,1.5);
\draw[shift={(10.5,0)}, color=black!30!white] (0,-0.25)--(1.5,-0.25)--(1.5,1.75)--(0,1.75)--cycle;
\draw[shift={(12.0,0)}] (0,0.0) to[out=0,in=90] (0.75,-0.25) to[out=90,in=180] (1.5,0.5);
\draw[shift={(12.0,0)}] (0,0.5) to[out=0,in=180] (1.5,0.0);
\draw[shift={(12.0,0)}] (0,1.0) to[out=0,in=180] (1.5,1.0);
\draw[shift={(12.0,0)}] (0,1.5) to[out=0,in=180] (1.5,1.5);
\draw[shift={(12.0,0)}, color=black!30!white] (0,-0.25)--(1.5,-0.25)--(1.5,1.75)--(0,1.75)--cycle;

\node at (-0.25,0.0) {\small $1$};
\node at (-0.25,0.5) {\small $2$};
\node at (-0.25,1.0) {\small $3$};
\node at (-0.25,1.5) {\small $4$};
\node at (13.75,0.0) {\small $3$};
\node at (13.75,0.5) {\small $1$};
\node at (13.75,1.0) {\small $4$};
\node at (13.75,1.5) {\small $2$};

\node[shift={( 0.0,0)}] at (0.75,2.05) {$C_3$};
\node[shift={( 1.5,0)}] at (0.75,2.05) {$C_3$};
\node[shift={( 3.0,0)}] at (0.75,2.05) {$C_2$};
\node[shift={( 4.5,0)}] at (0.75,2.05) {$C_4$};
\node[shift={( 6.0,0)}] at (0.75,2.05) {$C_3$};
\node[shift={( 7.5,0)}] at (0.75,2.05) {$C_4$};
\node[shift={( 9.0,0)}] at (0.75,2.05) {$C_3$};
\node[shift={(10.5,0)}] at (0.75,2.05) {$C_2$};
\node[shift={(12.0,0)}] at (0.75,2.05) {$C_2$};

\node[shift={( 0.0,0)}] at (0.75,-0.55) {$1$};
\node[shift={( 1.5,0)}] at (0.75,-0.55) {$2$};
\node[shift={( 3.0,0)}] at (0.75,-0.55) {$3$};
\node[shift={( 4.5,0)}] at (0.75,-0.55) {$1$};
\node[shift={( 6.0,0)}] at (0.75,-0.55) {$3$};
\node[shift={( 7.5,0)}] at (0.75,-0.55) {$2$};
\node[shift={( 9.0,0)}] at (0.75,-0.55) {$4$};
\node[shift={(10.5,0)}] at (0.75,-0.55) {$3$};
\node[shift={(12.0,0)}] at (0.75,-0.55) {$1$};
\end{tikzpicture}}

\def\picAD{
\begin{tikzpicture}[thick]
\draw[shift={( 0.0,0)}] (0,0.0) to[out=0,in=90] (0.6,-0.25) to[out=90,in=180] (1.5,0.5);
\draw[shift={( 0.0,0)}] (0,0.5) to[out=0,in=90] (0.9,-0.25) to[out=90,in=180] (1.5,2.0);
\draw[shift={( 0.0,0)}] (0,1.0) to[out=0,in=180] (1.5,0.0);
\draw[shift={( 0.0,0)}] (0,1.5) to[out=0,in=180] (1.5,1.0);
\draw[shift={( 0.0,0)}] (0,2.0) to[out=0,in=180] (1.5,1.5);
\draw[shift={( 0.0,0)}, color=black!30!white] (0,-0.25)--(1.5,-0.25)--(1.5,2.25)--(0,2.25)--cycle;
\node[shift={( 0.0,0)}] at (-0.25,0.0) {\small \emph1};
\node[shift={( 0.0,0)}] at (-0.25,0.5) {\small \emph2};
\node[shift={( 0.0,0)}] at (-0.25,1.0) {\small \emph3};
\node[shift={( 0.0,0)}] at (-0.25,1.5) {\small \emph4};
\node[shift={( 0.0,0)}] at (-0.25,2.0) {\small \emph5};
\node[shift={( 0.0,0)}] at ( 1.75,0.0) {\small \emph1};
\node[shift={( 0.0,0)}] at ( 1.75,0.5) {\small \emph2};
\node[shift={( 0.0,0)}] at ( 1.75,1.0) {\small \emph3};
\node[shift={( 0.0,0)}] at ( 1.75,1.5) {\small \emph4};
\node[shift={( 0.0,0)}] at ( 1.75,2.0) {\small \emph5};

\draw[shift={( 3.0,0)}] (0,0.0) to[out=0,in=90] (0.6,-0.25) to[out=90,in=180] (1.5,2.0);
\draw[shift={( 3.0,0)}] (0,0.5) to[out=0,in=90] (0.9,-0.25) to[out=90,in=180] (1.5,0.5);
\draw[shift={( 3.0,0)}] (0,1.0) to[out=0,in=180] (1.5,0.0);
\draw[shift={( 3.0,0)}] (0,1.5) to[out=0,in=180] (1.5,1.0);
\draw[shift={( 3.0,0)}] (0,2.0) to[out=0,in=180] (1.5,1.5);
\draw[shift={( 3.0,0)}, color=black!30!white] (0,-0.25)--(1.5,-0.25)--(1.5,2.25)--(0,2.25)--cycle;
\node[shift={( 3.0,0)}] at (-0.25,0.0) {\small \emph1};
\node[shift={( 3.0,0)}] at (-0.25,0.5) {\small \emph2};
\node[shift={( 3.0,0)}] at (-0.25,1.0) {\small \emph3};
\node[shift={( 3.0,0)}] at (-0.25,1.5) {\small \emph4};
\node[shift={( 3.0,0)}] at (-0.25,2.0) {\small \emph5};
\node[shift={( 3.0,0)}] at ( 1.75,0.0) {\small \emph1};
\node[shift={( 3.0,0)}] at ( 1.75,0.5) {\small \emph2};
\node[shift={( 3.0,0)}] at ( 1.75,1.0) {\small \emph3};
\node[shift={( 3.0,0)}] at ( 1.75,1.5) {\small \emph4};
\node[shift={( 3.0,0)}] at ( 1.75,2.0) {\small \emph5};

\node[shift={(0,0)}] at (0.75, 2.55) {$C_{2,5}$};
\node[shift={(3,0)}] at (0.75, 2.55) {$C_{5,2}$};
\end{tikzpicture}
}

\def\picAE{
\begin{tikzpicture}[thick]
\draw[shift={( 0.0,0)}] (0,0.0) to[out=0,in=90] (0.6,-0.25) to[out=90,in=180] (1.5,0.5);
\draw[shift={( 0.0,0)}] (0,0.5) to[out=0,in=90] (0.9,-0.25) to[out=90,in=180] (1.5,1.5);
\draw[shift={( 0.0,0)}] (0,1.0) to[out=0,in=180] (1.5,0.0);
\draw[shift={( 0.0,0)}] (0,1.5) to[out=0,in=180] (1.5,1.0);
\draw[shift={( 0.0,0)}] (0,2.0) to[out=0,in=180] (1.5,2.0);
\draw[shift={( 0.0,0)}, color=black!30!white] (0,-0.25)--(1.5,-0.25)--(1.5,2.25)--(0,2.25)--cycle;
\draw[shift={( 1.5,0)}] (0,0.0) to[out=0,in=90] (0.6,-0.25) to[out=90,in=180] (1.5,0.5);
\draw[shift={( 1.5,0)}] (0,0.5) to[out=0,in=90] (0.9,-0.25) to[out=90,in=180] (1.5,2.0);
\draw[shift={( 1.5,0)}] (0,1.0) to[out=0,in=180] (1.5,0.0);
\draw[shift={( 1.5,0)}] (0,1.5) to[out=0,in=180] (1.5,1.0);
\draw[shift={( 1.5,0)}] (0,2.0) to[out=0,in=180] (1.5,1.5);
\draw[shift={( 1.5,0)}, color=black!30!white] (0,-0.25)--(1.5,-0.25)--(1.5,2.25)--(0,2.25)--cycle;
\draw[shift={( 3.0,0)}] (0,0.0) to[out=0,in=90] (0.6,-0.25) to[out=90,in=180] (1.5,0.5);
\draw[shift={( 3.0,0)}] (0,0.5) to[out=0,in=90] (0.9,-0.25) to[out=90,in=180] (1.5,1.0);
\draw[shift={( 3.0,0)}] (0,1.0) to[out=0,in=180] (1.5,0.0);
\draw[shift={( 3.0,0)}] (0,1.5) to[out=0,in=180] (1.5,1.5);
\draw[shift={( 3.0,0)}] (0,2.0) to[out=0,in=180] (1.5,2.0);
\draw[shift={( 3.0,0)}, color=black!30!white] (0,-0.25)--(1.5,-0.25)--(1.5,2.25)--(0,2.25)--cycle;
\draw[shift={( 4.5,0)}] (0,0.0) to[out=0,in=90] (0.6,-0.25) to[out=90,in=180] (1.5,2.0);
\draw[shift={( 4.5,0)}] (0,0.5) to[out=0,in=90] (0.9,-0.25) to[out=90,in=180] (1.5,1.5);
\draw[shift={( 4.5,0)}] (0,1.0) to[out=0,in=180] (1.5,0.0);
\draw[shift={( 4.5,0)}] (0,1.5) to[out=0,in=180] (1.5,0.5);
\draw[shift={( 4.5,0)}] (0,2.0) to[out=0,in=180] (1.5,1.0);
\draw[shift={( 4.5,0)}, color=black!30!white] (0,-0.25)--(1.5,-0.25)--(1.5,2.25)--(0,2.25)--cycle;
\draw[shift={( 6.0,0)}] (0,0.0) to[out=0,in=90] (0.6,-0.25) to[out=90,in=180] (1.5,2.0);
\draw[shift={( 6.0,0)}] (0,0.5) to[out=0,in=90] (0.9,-0.25) to[out=90,in=180] (1.5,0.5);
\draw[shift={( 6.0,0)}] (0,1.0) to[out=0,in=180] (1.5,0.0);
\draw[shift={( 6.0,0)}] (0,1.5) to[out=0,in=180] (1.5,1.0);
\draw[shift={( 6.0,0)}] (0,2.0) to[out=0,in=180] (1.5,1.5);
\draw[shift={( 6.0,0)}, color=black!30!white] (0,-0.25)--(1.5,-0.25)--(1.5,2.25)--(0,2.25)--cycle;
\draw[shift={( 7.5,0)}] (0,0.0) to[out=0,in=90] (0.6,-0.25) to[out=90,in=180] (1.5,1.5);
\draw[shift={( 7.5,0)}] (0,0.5) to[out=0,in=90] (0.9,-0.25) to[out=90,in=180] (1.5,0.5);
\draw[shift={( 7.5,0)}] (0,1.0) to[out=0,in=180] (1.5,0.0);
\draw[shift={( 7.5,0)}] (0,1.5) to[out=0,in=180] (1.5,1.0);
\draw[shift={( 7.5,0)}] (0,2.0) to[out=0,in=180] (1.5,2.0);
\draw[shift={( 7.5,0)}, color=black!30!white] (0,-0.25)--(1.5,-0.25)--(1.5,2.25)--(0,2.25)--cycle;

\node[shift={( 0.0,0)}] at (-0.25,0.0) {\small 1};
\node[shift={( 0.0,0)}] at (-0.25,0.5) {\small 2};
\node[shift={( 0.0,0)}] at (-0.25,1.0) {\small 3};
\node[shift={( 0.0,0)}] at (-0.25,1.5) {\small 4};
\node[shift={( 0.0,0)}] at (-0.25,2.0) {\small 5};
\node[shift={( 0.0,0)}] at ( 9.25,0.0) {\small 4};
\node[shift={( 0.0,0)}] at ( 9.25,0.5) {\small 5};
\node[shift={( 0.0,0)}] at ( 9.25,1.0) {\small 2};
\node[shift={( 0.0,0)}] at ( 9.25,1.5) {\small 1};
\node[shift={( 0.0,0)}] at ( 9.25,2.0) {\small 3};

\node[shift={( 0.0,0)}] at (0.75,-0.55) {$1\,2$};
\node[shift={( 1.5,0)}] at (0.75,-0.55) {$3\,1$};
\node[shift={( 3.0,0)}] at (0.75,-0.55) {$4\,3$};
\node[shift={( 4.5,0)}] at (0.75,-0.55) {$2\,4$};
\node[shift={( 6.0,0)}] at (0.75,-0.55) {$3\,5$};
\node[shift={( 7.5,0)}] at (0.75,-0.55) {$1\,5$};

\node[shift={(0.0,0)}] at (0.75, 2.55) {$C_{2,4}$};
\node[shift={(1.5,0)}] at (0.75, 2.55) {$C_{2,5}$};
\node[shift={(3.0,0)}] at (0.75, 2.55) {$C_{2,3}$};
\node[shift={(4.5,0)}] at (0.75, 2.55) {$C_{5,4}$};
\node[shift={(6.0,0)}] at (0.75, 2.55) {$C_{5,2}$};
\node[shift={(7.5,0)}] at (0.75, 2.55) {$C_{4,2}$};

\end{tikzpicture}
}

\def\picAF{
\begin{tikzpicture}[thick, scale=2]
\node[draw,circle,inner sep=1pt] (x1) at (0,0) {$x_1$};
\node[draw,circle,inner sep=1pt] (x2) at (0,1) {$x_2$};
\node[draw,circle,inner sep=1pt] (x3) at (1,0) {$x_3$};
\node[draw,circle,inner sep=1pt] (x4) at (0.5,-0.5) {$x_4$};
\node[draw,circle,inner sep=1pt] (x5) at (1,1) {$x_5$};
\draw[->] (x2)--(x1);
\draw[->] (x1)--(x3);
\draw[->] (x3)--(x5);
\draw[->] (x5)--(x2);
\draw[->] (x1)--(x4);
\draw[->] (x3)--(x4);
\node at (0.5,0.125)  {$2$};
\node at (0.5,1.125)  {$4$};
\node at (-.1,0.5)  {$3$};
\node at (1.1,0.5)  {$1$};
\node at (0.25-0.08,-0.25-0.08) {$5$};
\node at (0.75+0.08,-0.25-0.08) {$6$};
\end{tikzpicture}}

\def\picAG{
\begin{tikzpicture}[thick]
\draw[shift={( 0.0,0)}] (0,0.0) to[out=0,in=90] (0.6,-0.25) to[out=90,in=180] (1.5,1.5);
\draw[shift={( 0.0,0)}] (0,0.5) to[out=0,in=90] (0.9,-0.25) to[out=90,in=180] (1.5,2.0);
\draw[shift={( 0.0,0)}] (0,1.0) to[out=0,in=180] (1.5,0.0);
\draw[shift={( 0.0,0)}] (0,1.5) to[out=0,in=180] (1.5,0.5);
\draw[shift={( 0.0,0)}] (0,2.0) to[out=0,in=180] (1.5,1.0);
\draw[shift={( 0.0,0)}, color=black!30!white] (0,-0.25)--(1.5,-0.25)--(1.5,2.25)--(0,2.25)--cycle;
\draw[shift={( 1.5,0)}] (0,0.0) to[out=0,in=90] (0.6,-0.25) to[out=90,in=180] (1.5,1.5);
\draw[shift={( 1.5,0)}] (0,0.5) to[out=0,in=90] (0.9,-0.25) to[out=90,in=180] (1.5,2.0);
\draw[shift={( 1.5,0)}] (0,1.0) to[out=0,in=180] (1.5,0.0);
\draw[shift={( 1.5,0)}] (0,1.5) to[out=0,in=180] (1.5,0.5);
\draw[shift={( 1.5,0)}] (0,2.0) to[out=0,in=180] (1.5,1.0);
\draw[shift={( 1.5,0)}, color=black!30!white] (0,-0.25)--(1.5,-0.25)--(1.5,2.25)--(0,2.25)--cycle;
\draw[shift={( 3.0,0)}] (0,0.0) to[out=0,in=90] (0.6,-0.25) to[out=90,in=180] (1.5,0.0);
\draw[shift={( 3.0,0)}] (0,0.5) to[out=0,in=90] (0.9,-0.25) to[out=90,in=180] (1.5,2.0);
\draw[shift={( 3.0,0)}] (0,1.0) to[out=0,in=180] (1.5,0.5);
\draw[shift={( 3.0,0)}] (0,1.5) to[out=0,in=180] (1.5,1.0);
\draw[shift={( 3.0,0)}] (0,2.0) to[out=0,in=180] (1.5,1.5);
\draw[shift={( 3.0,0)}, color=black!30!white] (0,-0.25)--(1.5,-0.25)--(1.5,2.25)--(0,2.25)--cycle;
\draw[shift={( 4.5,0)}] (0,0.0) to[out=0,in=90] (0.6,-0.25) to[out=90,in=180] (1.5,1.0);
\draw[shift={( 4.5,0)}] (0,0.5) to[out=0,in=90] (0.9,-0.25) to[out=90,in=180] (1.5,1.5);
\draw[shift={( 4.5,0)}] (0,1.0) to[out=0,in=180] (1.5,0.0);
\draw[shift={( 4.5,0)}] (0,1.5) to[out=0,in=180] (1.5,0.5);
\draw[shift={( 4.5,0)}] (0,2.0) to[out=0,in=180] (1.5,2.0);
\draw[shift={( 4.5,0)}, color=black!30!white] (0,-0.25)--(1.5,-0.25)--(1.5,2.25)--(0,2.25)--cycle;
\draw[shift={( 6.0,0)}] (0,0.0) to[out=0,in=90] (0.6,-0.25) to[out=90,in=180] (1.5,1.5);
\draw[shift={( 6.0,0)}] (0,0.5) to[out=0,in=90] (0.9,-0.25) to[out=90,in=180] (1.5,2.0);
\draw[shift={( 6.0,0)}] (0,1.0) to[out=0,in=180] (1.5,0.0);
\draw[shift={( 6.0,0)}] (0,1.5) to[out=0,in=180] (1.5,0.5);
\draw[shift={( 6.0,0)}] (0,2.0) to[out=0,in=180] (1.5,1.0);
\draw[shift={( 6.0,0)}, color=black!30!white] (0,-0.25)--(1.5,-0.25)--(1.5,2.25)--(0,2.25)--cycle;
\draw[shift={( 7.5,0)}] (0,0.0) to[out=0,in=90] (0.6,-0.25) to[out=90,in=180] (1.5,0.5);
\draw[shift={( 7.5,0)}] (0,0.5) to[out=0,in=90] (0.9,-0.25) to[out=90,in=180] (1.5,1.5);
\draw[shift={( 7.5,0)}] (0,1.0) to[out=0,in=180] (1.5,0.0);
\draw[shift={( 7.5,0)}] (0,1.5) to[out=0,in=180] (1.5,1.0);
\draw[shift={( 7.5,0)}] (0,2.0) to[out=0,in=180] (1.5,2.0);
\draw[shift={( 7.5,0)}, color=black!30!white] (0,-0.25)--(1.5,-0.25)--(1.5,2.25)--(0,2.25)--cycle;
\draw[shift={( 9.0,0)}] (0,0.0) to[out=0,in=90] (0.6,-0.25) to[out=90,in=180] (1.5,0.5);
\draw[shift={( 9.0,0)}] (0,0.5) to[out=0,in=90] (0.9,-0.25) to[out=90,in=180] (1.5,1.0);
\draw[shift={( 9.0,0)}] (0,1.0) to[out=0,in=180] (1.5,0.0);
\draw[shift={( 9.0,0)}] (0,1.5) to[out=0,in=180] (1.5,1.5);
\draw[shift={( 9.0,0)}] (0,2.0) to[out=0,in=180] (1.5,2.0);
\draw[shift={( 9.0,0)}, color=black!30!white] (0,-0.25)--(1.5,-0.25)--(1.5,2.25)--(0,2.25)--cycle;

\node at (-0.25,0.0) {\small $x_1$};
\node at (-0.25,0.5) {\small $x_3$};
\node at (-0.25,1.0) {\small $x_4$};
\node at (-0.25,1.5) {\small $x_2$};
\node at (-0.25,2.0) {\small $x_5$};
\node at (10.75,0.0) {\small $x_4$};
\node at (10.75,0.5) {\small $x_1$};
\node at (10.75,1.0) {\small $x_5$};
\node at (10.75,1.5) {\small $x_3$};
\node at (10.75,2.0) {\small $x_2$};

\node[shift={( 0.0,0)}] at (0.75,-0.55) {$\{x_1,x_3\}$};
\node[shift={( 1.5,0)}] at (0.75,-0.55) {$\{x_2,x_4\}$};
\node[shift={( 3.0,0)}] at (0.75,-0.55) {$\{x_1,x_5\}$};
\node[shift={( 4.5,0)}] at (0.75,-0.55) {$\{x_3,x_5\}$};
\node[shift={( 6.0,0)}] at (0.75,-0.55) {$\{x_2,x_4\}$};
\node[shift={( 7.5,0)}] at (0.75,-0.55) {$\{x_3,x_5\}$};
\node[shift={( 9.0,0)}] at (0.75,-0.55) {$\{x_1,x_5\}$};

\node[shift={(0.0,0)}] at (0.75, 2.55) {$C_{4,5}$};
\node[shift={(1.5,0)}] at (0.75, 2.55) {$C_{4,5}$};
\node[shift={(3.0,0)}] at (0.75, 2.55) {$C_{1,5}$};
\node[shift={(4.5,0)}] at (0.75, 2.55) {$C_{3,4}$};
\node[shift={(6.0,0)}] at (0.75, 2.55) {$C_{4,5}$};
\node[shift={(7.5,0)}] at (0.75, 2.55) {$C_{2,4}$};
\node[shift={(9.0,0)}] at (0.75, 2.55) {$C_{2,3}$};

\end{tikzpicture}}

\def\picAH{
\begin{tikzpicture}[thick]
\draw[shift={( 0.0,0)}] (0,0.0) to[out=0,in=90] (0.6,-0.25) to[out=90,in=180] (1.5,1.5);
\draw[shift={( 0.0,0)}] (0,0.5) to[out=0,in=90] (0.9,-0.25) to[out=90,in=180] (1.5,2.0);
\draw[shift={( 0.0,0)}] (0,1.0) to[out=0,in=180] (1.5,0.0);
\draw[shift={( 0.0,0)}] (0,1.5) to[out=0,in=180] (1.5,0.5);
\draw[shift={( 0.0,0)}] (0,2.0) to[out=0,in=180] (1.5,1.0);
\draw[shift={( 0.0,0)}, color=black!30!white] (0,-0.25)--(1.5,-0.25)--(1.5,2.25)--(0,2.25)--cycle;
\draw[shift={( 1.5,0)}] (0,0.0) to[out=0,in=90] (0.6,-0.25) to[out=90,in=180] (1.5,1.5);
\draw[shift={( 1.5,0)}] (0,0.5) to[out=0,in=90] (0.9,-0.25) to[out=90,in=180] (1.5,2.0);
\draw[shift={( 1.5,0)}] (0,1.0) to[out=0,in=180] (1.5,0.0);
\draw[shift={( 1.5,0)}] (0,1.5) to[out=0,in=180] (1.5,0.5);
\draw[shift={( 1.5,0)}] (0,2.0) to[out=0,in=180] (1.5,1.0);
\draw[shift={( 1.5,0)}, color=black!30!white] (0,-0.25)--(1.5,-0.25)--(1.5,2.25)--(0,2.25)--cycle;
\draw[shift={( 3.0,0)}] (0,0.0) to[out=0,in=90] (0.6,-0.25) to[out=90,in=180] (1.5,0.0);
\draw[shift={( 3.0,0)}] (0,0.5) to[out=0,in=90] (0.9,-0.25) to[out=90,in=180] (1.5,2.0);
\draw[shift={( 3.0,0)}] (0,1.0) to[out=0,in=180] (1.5,0.5);
\draw[shift={( 3.0,0)}] (0,1.5) to[out=0,in=180] (1.5,1.0);
\draw[shift={( 3.0,0)}] (0,2.0) to[out=0,in=180] (1.5,1.5);
\draw[shift={( 3.0,0)}, color=black!30!white] (0,-0.25)--(1.5,-0.25)--(1.5,2.25)--(0,2.25)--cycle;
\draw[shift={( 4.5,0)}] (0,0.0) to[out=0,in=90] (0.6,-0.25) to[out=90,in=180] (1.5,1.0);
\draw[shift={( 4.5,0)}] (0,0.5) to[out=0,in=90] (0.9,-0.25) to[out=90,in=180] (1.5,1.5);
\draw[shift={( 4.5,0)}] (0,1.0) to[out=0,in=180] (1.5,0.0);
\draw[shift={( 4.5,0)}] (0,1.5) to[out=0,in=180] (1.5,0.5);
\draw[shift={( 4.5,0)}] (0,2.0) to[out=0,in=180] (1.5,2.0);
\draw[shift={( 4.5,0)}, color=black!30!white] (0,-0.25)--(1.5,-0.25)--(1.5,2.25)--(0,2.25)--cycle;
\draw[shift={( 6.0,0)}] (0,0.0) to[out=0,in=90] (0.6,-0.25) to[out=90,in=180] (1.5,1.5);
\draw[shift={( 6.0,0)}] (0,0.5) to[out=0,in=90] (0.9,-0.25) to[out=90,in=180] (1.5,2.0);
\draw[shift={( 6.0,0)}] (0,1.0) to[out=0,in=180] (1.5,0.0);
\draw[shift={( 6.0,0)}] (0,1.5) to[out=0,in=180] (1.5,0.5);
\draw[shift={( 6.0,0)}] (0,2.0) to[out=0,in=180] (1.5,1.0);
\draw[shift={( 6.0,0)}, color=black!30!white] (0,-0.25)--(1.5,-0.25)--(1.5,2.25)--(0,2.25)--cycle;
\draw[shift={( 7.5,0)}] (0,0.0) to[out=0,in=90] (0.6,-0.25) to[out=90,in=180] (1.5,0.5);
\draw[shift={( 7.5,0)}] (0,0.5) to[out=0,in=90] (0.9,-0.25) to[out=90,in=180] (1.5,1.5);
\draw[shift={( 7.5,0)}] (0,1.0) to[out=0,in=180] (1.5,0.0);
\draw[shift={( 7.5,0)}] (0,1.5) to[out=0,in=180] (1.5,1.0);
\draw[shift={( 7.5,0)}] (0,2.0) to[out=0,in=180] (1.5,2.0);
\draw[shift={( 7.5,0)}, color=black!30!white] (0,-0.25)--(1.5,-0.25)--(1.5,2.25)--(0,2.25)--cycle;
\draw[shift={( 9.0,0)}] (0,0.0) to[out=0,in=90] (0.6,-0.25) to[out=90,in=180] (1.5,0.5);
\draw[shift={( 9.0,0)}] (0,0.5) to[out=0,in=90] (0.9,-0.25) to[out=90,in=180] (1.5,1.0);
\draw[shift={( 9.0,0)}] (0,1.0) to[out=0,in=180] (1.5,0.0);
\draw[shift={( 9.0,0)}] (0,1.5) to[out=0,in=180] (1.5,1.5);
\draw[shift={( 9.0,0)}] (0,2.0) to[out=0,in=180] (1.5,2.0);
\draw[shift={( 9.0,0)}, color=black!30!white] (0,-0.25)--(1.5,-0.25)--(1.5,2.25)--(0,2.25)--cycle;

\node at (-0.25,0.0) {\small $1$};
\node at (-0.25,0.5) {\small $2$};
\node at (-0.25,1.0) {\small $3$};
\node at (-0.25,1.5) {\small $4$};
\node at (-0.25,2.0) {\small $5$};
\node at (10.75,0.0) {\small $3$};
\node at (10.75,0.5) {\small $1$};
\node at (10.75,1.0) {\small $5$};
\node at (10.75,1.5) {\small $2$};
\node at (10.75,2.0) {\small $4$};

\node[shift={( 0.0,0)}] at (0.75,-0.55) {$1\,2$};
\node[shift={( 1.5,0)}] at (0.75,-0.55) {$3\,4$};
\node[shift={( 3.0,0)}] at (0.75,-0.55) {$5\,1$};
\node[shift={( 4.5,0)}] at (0.75,-0.55) {$5\,2$};
\node[shift={( 6.0,0)}] at (0.75,-0.55) {$3\,4$};
\node[shift={( 7.5,0)}] at (0.75,-0.55) {$5\,2$};
\node[shift={( 9.0,0)}] at (0.75,-0.55) {$1\,5$};

\node[shift={(0.0,0)}] at (0.75, 2.55) {$C_{4,5}$};
\node[shift={(1.5,0)}] at (0.75, 2.55) {$C_{4,5}$};
\node[shift={(3.0,0)}] at (0.75, 2.55) {$C_{1,5}$};
\node[shift={(4.5,0)}] at (0.75, 2.55) {$C_{3,4}$};
\node[shift={(6.0,0)}] at (0.75, 2.55) {$C_{4,5}$};
\node[shift={(7.5,0)}] at (0.75, 2.55) {$C_{2,4}$};
\node[shift={(9.0,0)}] at (0.75, 2.55) {$C_{2,3}$};

\end{tikzpicture}}

\def\picBA{
\begin{tikzpicture}[thick, scale=1.5]
\node[draw,circle,inner sep=1pt] (x1) at (0,0) {$x_1$};
\node[draw,circle,inner sep=1pt] (x2) at (1,0) {$x_2$};
\draw (x1) to[out=20,in=180] (0.5, 0.125) to[out=0,in=160] (x2);
\draw (x1) to[out=-20,in=180] (0.5, -0.125) to[out=0,in=200] (x2);
\node at (0.5,0.275)  {$1$};
\node at (0.5,-0.275)  {$2$};
\end{tikzpicture}}

\def\picBB{\begin{tikzpicture}[thick, scale=1.5]
\node[draw,circle,inner sep=1pt] (x1) at (0,0) {$x_1$};
\node[draw,circle,inner sep=1pt] (x2) at (1,0) {$x_2$};
\node[draw,circle,inner sep=1pt] (x3) at (2,0) {$x_3$};
\draw (x1)--(x2)--(x3);
\node at (0.5,0.15)  {$1$};
\node at (1.5,0.15)  {$2$};
\end{tikzpicture}}

\def\picBC{
\begin{tikzpicture}[thick, scale=1.5]
\node[draw,circle,inner sep=1pt] (x1) at (0,0) {$x_1$};
\node[draw,circle,inner sep=1pt] (x2) at (1,0) {$x_2$};
\node[draw,circle,inner sep=1pt] (x3) at (0,0.75) {$x_3$};
\node[draw,circle,inner sep=1pt] (x4) at (1,0.75) {$x_4$};
\draw (x1)--(x2);
\draw (x3)--(x4);
\node at (0.5,0.15)  {$1$};
\node at (0.5,0.9)  {$2$};
\end{tikzpicture}}

\def\picBD{
\begin{tikzpicture}[thick]

\draw[shift={( 0.0,0)}] (0,0.0) to[out=0,in=90] (0.6,-0.25) to[out=90,in=180] (1.5,0.0);
\draw[shift={( 0.0,0)}] (0,0.5) to[out=0,in=90] (0.9,-0.25) to[out=90,in=180] (1.5,0.5);
\draw[shift={( 0.0,0)}, color=black!30!white] (0,-0.25)--(1.5,-0.25)--(1.5,0.75)--(0,0.75)--cycle;
\draw[shift={( 1.5,0)}] (0,0.0) to[out=0,in=90] (0.6,-0.25) to[out=90,in=180] (1.5,0.0);
\draw[shift={( 1.5,0)}] (0,0.5) to[out=0,in=90] (0.9,-0.25) to[out=90,in=180] (1.5,0.5);
\draw[shift={( 1.5,0)}, color=black!30!white] (0,-0.25)--(1.5,-0.25)--(1.5,0.75)--(0,0.75)--cycle;

\node at (-0.25,0.0) {\small \emph1};
\node at (-0.25,0.5) {\small \emph2};
\node at ( 3.25,0.0) {\small \emph1};
\node at ( 3.25,0.5) {\small \emph2};

\node[shift={( 0.0,0)}] at (0.75,-0.55) {$1\,2$};
\node[shift={( 1.5,0)}] at (0.75,-0.55) {$1\,2$};
\end{tikzpicture}
}

\def\picBE{\begin{tikzpicture}[thick]

\draw[shift={( 0.0,0)}] (0,0.0) to[out=0,in=90] (0.6,-0.25) to[out=90,in=180] (1.5,0.0);
\draw[shift={( 0.0,0)}] (0,0.5) to[out=0,in=90] (0.9,-0.25) to[out=90,in=180] (1.5,1.0);
\draw[shift={( 0.0,0)}] (0,1.0) to[out=0,in=180] (1.5,0.5);
\draw[shift={( 0.0,0)}, color=black!30!white] (0,-0.25)--(1.5,-0.25)--(1.5,1.25)--(0,1.25)--cycle;
\draw[shift={( 1.5,0)}] (0,0.0) to[out=0,in=90] (0.6,-0.25) to[out=90,in=180] (1.5,0.0);
\draw[shift={( 1.5,0)}] (0,0.5) to[out=0,in=90] (0.9,-0.25) to[out=90,in=180] (1.5,1.0);
\draw[shift={( 1.5,0)}] (0,1.0) to[out=0,in=180] (1.5,0.5);
\draw[shift={( 1.5,0)}, color=black!30!white] (0,-0.25)--(1.5,-0.25)--(1.5,1.25)--(0,1.25)--cycle;

\node at (-0.25,0.0) {\small \emph1};
\node at (-0.25,0.5) {\small \emph2};
\node at (-0.25,1.0) {\small \emph3};
\node at ( 3.25,0.0) {\small \emph1};
\node at ( 3.25,0.5) {\small \emph2};
\node at ( 3.25,1.0) {\small \emph3};

\node[shift={( 0.0,0)}] at (0.75,-0.55) {$1\,2$};
\node[shift={( 1.5,0)}] at (0.75,-0.55) {$1\,3$};
\end{tikzpicture}
}

\def\picBF{\begin{tikzpicture}[thick]

\draw[shift={( 0.0,0)}] (0,0.0) to[out=0,in=90] (0.6,-0.25) to[out=90,in=180] (1.5,1.0);
\draw[shift={( 0.0,0)}] (0,0.5) to[out=0,in=90] (0.9,-0.25) to[out=90,in=180] (1.5,1.5);
\draw[shift={( 0.0,0)}] (0,1.0) to[out=0,in=180] (1.5,0.0);
\draw[shift={( 0.0,0)}] (0,1.5) to[out=0,in=180] (1.5,0.5);
\draw[shift={( 0.0,0)}, color=black!30!white] (0,-0.25)--(1.5,-0.25)--(1.5,1.75)--(0,1.75)--cycle;
\draw[shift={( 1.5,0)}] (0,0.0) to[out=0,in=90] (0.6,-0.25) to[out=90,in=180] (1.5,1.0);
\draw[shift={( 1.5,0)}] (0,0.5) to[out=0,in=90] (0.9,-0.25) to[out=90,in=180] (1.5,1.5);
\draw[shift={( 1.5,0)}] (0,1.0) to[out=0,in=180] (1.5,0.0);
\draw[shift={( 1.5,0)}] (0,1.5) to[out=0,in=180] (1.5,0.5);
\draw[shift={( 1.5,0)}, color=black!30!white] (0,-0.25)--(1.5,-0.25)--(1.5,1.75)--(0,1.75)--cycle;

\node at (-0.25,0.0) {\small \emph1};
\node at (-0.25,0.5) {\small \emph2};
\node at (-0.25,1.0) {\small \emph3};
\node at (-0.25,1.5) {\small \emph4};
\node at ( 3.25,0.0) {\small \emph1};
\node at ( 3.25,0.5) {\small \emph2};
\node at ( 3.25,1.0) {\small \emph3};
\node at ( 3.25,1.5) {\small \emph4};

\node[shift={( 0.0,0)}] at (0.75,-0.55) {$1\,2$};
\node[shift={( 1.5,0)}] at (0.75,-0.55) {$3\,4$};
\end{tikzpicture}}

\def\picZZ{
\begin{tikzpicture}[thick]
\draw[shift={( 0.0,0)}] (0,0.0) to[out=0,in=90] (0.6,-0.25) to[out=90,in=180] (1.5,1.0);
\draw[shift={( 0.0,0)}] (0,0.5) to[out=0,in=90] (0.9,-0.25) to[out=90,in=180] (1.5,0.0);
\draw[shift={( 0.0,0)}] (0,1.0) to[out=0,in=180]                              (1.5,0.5);
\draw[shift={( 0.0,0)}] (0,1.5) to[out=0,in=180]                              (1.5,1.5);
\draw[shift={( 0.0,0)}] (0,2.0) to[out=0,in=180]                              (1.5,2.0);
\draw[shift={( 0.0,0)}, color=black!30!white] (0,-0.25)--(1.5,-0.25)--(1.5,2.25)--(0,2.25)--cycle;
\draw[shift={( 1.5,0)}] (0,0.0) to[out=0,in=90] (0.6,-0.25) to[out=90,in=180] (1.5,0.5);
\draw[shift={( 1.5,0)}] (0,0.5) to[out=0,in=90] (0.9,-0.25) to[out=90,in=180] (1.5,2.0);
\draw[shift={( 1.5,0)}] (0,1.0) to[out=0,in=180]                              (1.5,0.0);
\draw[shift={( 1.5,0)}] (0,1.5) to[out=0,in=180]                              (1.5,1.0);
\draw[shift={( 1.5,0)}] (0,2.0) to[out=0,in=180]                              (1.5,1.5);
\draw[shift={( 1.5,0)}, color=black!30!white] (0,-0.25)--(1.5,-0.25)--(1.5,2.25)--(0,2.25)--cycle;
\draw[shift={( 3.0,0)}] (0,0.0) to[out=0,in=90] (0.6,-0.25) to[out=90,in=180] (1.5,0.0);
\draw[shift={( 3.0,0)}] (0,0.5) to[out=0,in=90] (0.9,-0.25) to[out=90,in=180] (1.5,1.5);
\draw[shift={( 3.0,0)}] (0,1.0) to[out=0,in=180]                              (1.5,0.5);
\draw[shift={( 3.0,0)}] (0,1.5) to[out=0,in=180]                              (1.5,1.0);
\draw[shift={( 3.0,0)}] (0,2.0) to[out=0,in=180]                              (1.5,2.0);
\draw[shift={( 3.0,0)}, color=black!30!white] (0,-0.25)--(1.5,-0.25)--(1.5,2.25)--(0,2.25)--cycle;
\draw[shift={( 4.5,0)}] (0,0.0) to[out=0,in=90] (0.6,-0.25) to[out=90,in=180] (1.5,1.5);
\draw[shift={( 4.5,0)}] (0,0.5) to[out=0,in=90] (0.9,-0.25) to[out=90,in=180] (1.5,2.0);
\draw[shift={( 4.5,0)}] (0,1.0) to[out=0,in=180]                              (1.5,0.0);
\draw[shift={( 4.5,0)}] (0,1.5) to[out=0,in=180]                              (1.5,0.5);
\draw[shift={( 4.5,0)}] (0,2.0) to[out=0,in=180]                              (1.5,1.0);
\draw[shift={( 4.5,0)}, color=black!30!white] (0,-0.25)--(1.5,-0.25)--(1.5,2.25)--(0,2.25)--cycle;
\draw[shift={( 6.0,0)}] (0,0.0) to[out=0,in=90] (0.6,-0.25) to[out=90,in=180] (1.5,2.0);
\draw[shift={( 6.0,0)}] (0,0.5) to[out=0,in=90] (0.9,-0.25) to[out=90,in=180] (1.5,0.5);
\draw[shift={( 6.0,0)}] (0,1.0) to[out=0,in=180]                              (1.5,0.0);
\draw[shift={( 6.0,0)}] (0,1.5) to[out=0,in=180]                              (1.5,1.0);
\draw[shift={( 6.0,0)}] (0,2.0) to[out=0,in=180]                              (1.5,1.5);
\draw[shift={( 6.0,0)}, color=black!30!white] (0,-0.25)--(1.5,-0.25)--(1.5,2.25)--(0,2.25)--cycle;
\draw[shift={( 7.5,0)}] (0,0.0) to[out=0,in=90] (0.6,-0.25) to[out=90,in=180] (1.5,0.0);
\draw[shift={( 7.5,0)}] (0,0.5) to[out=0,in=90] (0.9,-0.25) to[out=90,in=180] (1.5,0.5);
\draw[shift={( 7.5,0)}] (0,1.0) to[out=0,in=180]                              (1.5,1.0);
\draw[shift={( 7.5,0)}] (0,1.5) to[out=0,in=180]                              (1.5,1.5);
\draw[shift={( 7.5,0)}] (0,2.0) to[out=0,in=180]                              (1.5,2.0);
\draw[shift={( 7.5,0)}, color=black!30!white] (0,-0.25)--(1.5,-0.25)--(1.5,2.25)--(0,2.25)--cycle;
\draw[shift={( 9.0,0)}] (0,0.0) to[out=0,in=90] (0.6,-0.25) to[out=90,in=180] (1.5,0.0);
\draw[shift={( 9.0,0)}] (0,0.5) to[out=0,in=90] (0.9,-0.25) to[out=90,in=180] (1.5,1.0);
\draw[shift={( 9.0,0)}] (0,1.0) to[out=0,in=180]                              (1.5,0.5);
\draw[shift={( 9.0,0)}] (0,1.5) to[out=0,in=180]                              (1.5,1.5);
\draw[shift={( 9.0,0)}] (0,2.0) to[out=0,in=180]                              (1.5,2.0);
\draw[shift={( 9.0,0)}, color=black!30!white] (0,-0.25)--(1.5,-0.25)--(1.5,2.25)--(0,2.25)--cycle;
\draw[shift={(10.5,0)}] (0,0.0) to[out=0,in=90] (0.6,-0.25) to[out=90,in=180] (1.5,1.0);
\draw[shift={(10.5,0)}] (0,0.5) to[out=0,in=90] (0.9,-0.25) to[out=90,in=180] (1.5,0.0);
\draw[shift={(10.5,0)}] (0,1.0) to[out=0,in=180]                              (1.5,0.5);
\draw[shift={(10.5,0)}] (0,1.5) to[out=0,in=180]                              (1.5,1.5);
\draw[shift={(10.5,0)}] (0,2.0) to[out=0,in=180]                              (1.5,2.0);
\draw[shift={(10.5,0)}, color=black!30!white] (0,-0.25)--(1.5,-0.25)--(1.5,2.25)--(0,2.25)--cycle;
\draw[shift={(12.0,0)}] (0,0.0) to[out=0,in=90] (0.6,-0.25) to[out=90,in=180] (1.5,2.0);
\draw[shift={(12.0,0)}] (0,0.5) to[out=0,in=90] (0.9,-0.25) to[out=90,in=180] (1.5,1.0);
\draw[shift={(12.0,0)}] (0,1.0) to[out=0,in=180]                              (1.5,0.0);
\draw[shift={(12.0,0)}] (0,1.5) to[out=0,in=180]                              (1.5,0.5);
\draw[shift={(12.0,0)}] (0,2.0) to[out=0,in=180]                              (1.5,1.5);
\draw[shift={(12.0,0)}, color=black!30!white] (0,-0.25)--(1.5,-0.25)--(1.5,2.25)--(0,2.25)--cycle;

\node[shift={( 0.0,0)}] at (-0.25,0.0) {\small 1};
\node[shift={( 0.0,0)}] at (-0.25,0.5) {\small 2};
\node[shift={( 0.0,0)}] at (-0.25,1.0) {\small 3};
\node[shift={( 0.0,0)}] at (-0.25,1.5) {\small 4};
\node[shift={( 0.0,0)}] at (-0.25,2.0) {\small 5};
\node[shift={( 0.0,0)}] at (13.75,0.0) {\small 3};
\node[shift={( 0.0,0)}] at (13.75,0.5) {\small 4};
\node[shift={( 0.0,0)}] at (13.75,1.0) {\small 2};
\node[shift={( 0.0,0)}] at (13.75,1.5) {\small 5};
\node[shift={( 0.0,0)}] at (13.75,2.0) {\small 1};

\node[shift={( 0.0,0)}] at (0.75,-0.55) {$1\,2$};
\node[shift={( 1.5,0)}] at (0.75,-0.55) {$2\,3$};
\node[shift={( 3.0,0)}] at (0.75,-0.55) {$1\,2$};
\node[shift={( 4.5,0)}] at (0.75,-0.55) {$1\,4$};
\node[shift={( 6.0,0)}] at (0.75,-0.55) {$5\,2$};
\node[shift={( 7.5,0)}] at (0.75,-0.55) {$3\,2$};
\node[shift={( 9.0,0)}] at (0.75,-0.55) {$3\,2$};
\node[shift={(10.5,0)}] at (0.75,-0.55) {$3\,1$};
\node[shift={(12.0,0)}] at (0.75,-0.55) {$1\,2$};

\node[shift={( 0.0,0)}] at (0.75, 2.55) {$C_{3,1}$};
\node[shift={( 1.5,0)}] at (0.75, 2.55) {$C_{2,5}$};
\node[shift={( 3.0,0)}] at (0.75, 2.55) {$C_{1,4}$};
\node[shift={( 4.5,0)}] at (0.75, 2.55) {$C_{4,5}$};
\node[shift={( 6.0,0)}] at (0.75, 2.55) {$C_{5,2}$};
\node[shift={( 7.5,0)}] at (0.75, 2.55) {$C_{1,2}$};
\node[shift={( 9.0,0)}] at (0.75, 2.55) {$C_{1,3}$};
\node[shift={(10.5,0)}] at (0.75, 2.55) {$C_{3,1}$};
\node[shift={(12.0,0)}] at (0.75, 2.55) {$C_{5,3}$};
\end{tikzpicture}}

\def\picDA{
\begin{tikzpicture}[thick]

\draw[shift={( 0.0,0)}] (0,0.0) to[out=0,in=90] (0.75,-0.25) to[out=90,in=180] (1.5,0.5);
\draw[shift={( 0.0,0)}] (0,0.5) to[out=0,in=180]                              (1.5,0.0);
\draw[shift={( 0.0,0)}, color=black!30!white] (0,-0.25)--(1.5,-0.25)--(1.5,0.75)--(0,0.75)--cycle;
\draw[shift={( 1.5,0)}] (0,0.0) to[out=0,in=90] (0.75,-0.25) to[out=90,in=180] (1.5,0.0);
\draw[shift={( 1.5,0)}] (0,0.5) to[out=0,in=180]                              (1.5,0.5);
\draw[shift={( 1.5,0)}, color=black!30!white] (0,-0.25)--(1.5,-0.25)--(1.5,0.75)--(0,0.75)--cycle;
\draw[shift={( 3.0,0)}] (0,0.0) to[out=0,in=90] (0.75,-0.25) to[out=90,in=180] (1.5,0.5);
\draw[shift={( 3.0,0)}] (0,0.5) to[out=0,in=180]                              (1.5,0.0);
\draw[shift={( 3.0,0)}, color=black!30!white] (0,-0.25)--(1.5,-0.25)--(1.5,0.75)--(0,0.75)--cycle;

\node at (-0.25,0.0) {\small \emph1};
\node at (-0.25,0.5) {\small \emph2};
\node at ( 4.75,0.0) {\small \emph1};
\node at ( 4.75,0.5) {\small \emph2};

\node[shift={( 0.0,0)}] at (0.75,1.05) {$C_2$};
\node[shift={( 1.5,0)}] at (0.75,1.05) {$C_1$};
\node[shift={( 3.0,0)}] at (0.75,1.05) {$C_2$};
\end{tikzpicture}
}

\def\picDB{
\begin{tikzpicture}[thick]

\draw[shift={( 0.0,0)}] (0,0.0) to[out=0,in=90] (0.75,-0.25) to[out=90,in=180] (1.5,0.5);
\draw[shift={( 0.0,0)}] (0,0.5) to[out=0,in=180]                              (1.5,0.0);
\draw[shift={( 0.0,0)}] (0,1.0) to[out=0,in=180]                              (1.5,1.0);
\draw[shift={( 0.0,0)}, color=black!30!white] (0,-0.25)--(1.5,-0.25)--(1.5,1.25)--(0,1.25)--cycle;
\draw[shift={( 1.5,0)}] (0,0.0) to[out=0,in=90] (0.75,-0.25) to[out=90,in=180] (1.5,1.0);
\draw[shift={( 1.5,0)}] (0,0.5) to[out=0,in=180]                              (1.5,0.0);
\draw[shift={( 1.5,0)}] (0,1.0) to[out=0,in=180]                              (1.5,0.5);
\draw[shift={( 1.5,0)}, color=black!30!white] (0,-0.25)--(1.5,-0.25)--(1.5,1.25)--(0,1.25)--cycle;
\draw[shift={( 3.0,0)}] (0,0.0) to[out=0,in=90] (0.75,-0.25) to[out=90,in=180] (1.5,0.5);
\draw[shift={( 3.0,0)}] (0,0.5) to[out=0,in=180]                              (1.5,0.0);
\draw[shift={( 3.0,0)}] (0,1.0) to[out=0,in=180]                              (1.5,1.0);
\draw[shift={( 3.0,0)}, color=black!30!white] (0,-0.25)--(1.5,-0.25)--(1.5,1.25)--(0,1.25)--cycle;
\draw[shift={( 4.5,0)}] (0,0.0) to[out=0,in=90] (0.75,-0.25) to[out=90,in=180] (1.5,1.0);
\draw[shift={( 4.5,0)}] (0,0.5) to[out=0,in=180]                              (1.5,0.0);
\draw[shift={( 4.5,0)}] (0,1.0) to[out=0,in=180]                              (1.5,0.5);
\draw[shift={( 4.5,0)}, color=black!30!white] (0,-0.25)--(1.5,-0.25)--(1.5,1.25)--(0,1.25)--cycle;

\node at (-0.25,0.0) {\small \emph1};
\node at (-0.25,0.5) {\small \emph2};
\node at (-0.25,1.0) {\small \emph3};
\node at ( 6.25,0.0) {\small \emph1};
\node at ( 6.25,0.5) {\small \emph2};
\node at ( 6.25,1.0) {\small \emph3};

\node[shift={( 0.0,0)}] at (0.75,1.55) {$C_2$};
\node[shift={( 1.5,0)}] at (0.75,1.55) {$C_3$};
\node[shift={( 3.0,0)}] at (0.75,1.55) {$C_2$};
\node[shift={( 4.5,0)}] at (0.75,1.55) {$C_3$};
\end{tikzpicture}
}

\def\picDC{
\begin{tikzpicture}[thick]

\draw[shift={( 0.0,0)}] (0,0.0) to[out=0,in=90] (0.75,-0.25) to[out=90,in=180] (1.5,1.0);
\draw[shift={( 0.0,0)}] (0,0.5) to[out=0,in=180]                              (1.5,0.0);
\draw[shift={( 0.0,0)}] (0,1.0) to[out=0,in=180]                              (1.5,0.5);
\draw[shift={( 0.0,0)}] (0,1.5) to[out=0,in=180]                              (1.5,1.5);
\draw[shift={( 0.0,0)}] (0,2.0) to[out=0,in=180]                              (1.5,2.0);
\draw[shift={( 0.0,0)}, color=black!30!white] (0,-0.25)--(1.5,-0.25)--(1.5,2.25)--(0,2.25)--cycle;

\fill[shift={( 1.5,0)}, color=black!15!white] (0,1.0)--(0,2.0) to[out=0,in=180] (1.5,1.5)--(1.5,0.5) to[out=180,in=0] (0,1.0)--cycle; 
\draw[shift={( 1.5,0)}] (0,0.0) to[out=0,in=90] (0.75,-0.25) to[out=90,in=180] (1.5,2.0);
\draw[shift={( 1.5,0)}] (0,0.5) to[out=0,in=180]                              (1.5,0.0);
\draw[shift={( 1.5,0)}] (0,1.0) to[out=0,in=180]                              (1.5,0.5);
\draw[shift={( 1.5,0)}] (0,1.5) to[out=0,in=180]                              (1.5,1.0);
\draw[shift={( 1.5,0)}] (0,2.0) to[out=0,in=180]                              (1.5,1.5);
\draw[shift={( 1.5,0)}, color=black!30!white] (0,-0.25)--(1.5,-0.25)--(1.5,2.25)--(0,2.25)--cycle;

\fill[shift={( 3.0,0)}, color=black!15!white] (0,0.5)--(0,1.5) to[out=0,in=180] (1.5,1.5)--(1.5,0.5) to[out=180,in=0] (0,0.5)--cycle; 
\draw[shift={( 3.0,0)}] (0,0.0) to[out=0,in=90] (0.75,-0.25) to[out=90,in=180] (1.5,0.0);
\draw[shift={( 3.0,0)}] (0,0.5) to[out=0,in=180]                              (1.5,0.5);
\draw[shift={( 3.0,0)}] (0,1.0) to[out=0,in=180]                              (1.5,1.0);
\draw[shift={( 3.0,0)}] (0,1.5) to[out=0,in=180]                              (1.5,1.5);
\draw[shift={( 3.0,0)}] (0,2.0) to[out=0,in=180]                              (1.5,2.0);
\draw[shift={( 3.0,0)}, color=black!30!white] (0,-0.25)--(1.5,-0.25)--(1.5,2.25)--(0,2.25)--cycle;

\fill[shift={( 4.5,0)}, color=black!15!white] (0,0.5)--(0,1.5) to[out=0,in=180] (1.5,1.0)--(1.5,0.0) to[out=180,in=0] (0,0.5)--cycle; 
\draw[shift={( 4.5,0)}] (0,0.0) to[out=0,in=90] (0.75,-0.25) to[out=90,in=180] (1.5,2.0);
\draw[shift={( 4.5,0)}] (0,0.5) to[out=0,in=180]                              (1.5,0.0);
\draw[shift={( 4.5,0)}] (0,1.0) to[out=0,in=180]                              (1.5,0.5);
\draw[shift={( 4.5,0)}] (0,1.5) to[out=0,in=180]                              (1.5,1.0);
\draw[shift={( 4.5,0)}] (0,2.0) to[out=0,in=180]                              (1.5,1.5);
\draw[shift={( 4.5,0)}, color=black!30!white] (0,-0.25)--(1.5,-0.25)--(1.5,2.25)--(0,2.25)--cycle;

\fill[shift={( 6.0,0)}, color=black!15!white] (0,1.5)--(0,2.0) to[out=0,in=180] (1.5,2.0)--(1.5,1.5) to[out=180,in=0] (0,1.5)--cycle; 
\draw[shift={( 6.0,0)}] (0,0.0) to[out=0,in=90] (0.75,-0.25) to[out=90,in=180] (1.5,0.5);
\draw[shift={( 6.0,0)}] (0,0.5) to[out=0,in=180]                              (1.5,0.0);
\draw[shift={( 6.0,0)}] (0,1.0) to[out=0,in=180]                              (1.5,1.0);
\draw[shift={( 6.0,0)}] (0,1.5) to[out=0,in=180]                              (1.5,1.5);
\draw[shift={( 6.0,0)}] (0,2.0) to[out=0,in=180]                              (1.5,2.0);
\draw[shift={( 6.0,0)}, color=black!30!white] (0,-0.25)--(1.5,-0.25)--(1.5,2.25)--(0,2.25)--cycle;

\fill[shift={( 7.5,0)}, color=black!15!white] (0,1.5)--(0,2.0) to[out=0,in=180] (1.5,1.5)--(1.5,1.0) to[out=180,in=0] (0,1.5)--cycle; 
\draw[shift={( 7.5,0)}] (0,0.0) to[out=0,in=90] (0.75,-0.25) to[out=90,in=180] (1.5,2.0);
\draw[shift={( 7.5,0)}] (0,0.5) to[out=0,in=180]                              (1.5,0.0);
\draw[shift={( 7.5,0)}] (0,1.0) to[out=0,in=180]                              (1.5,0.5);
\draw[shift={( 7.5,0)}] (0,1.5) to[out=0,in=180]                              (1.5,1.0);
\draw[shift={( 7.5,0)}] (0,2.0) to[out=0,in=180]                              (1.5,1.5);
\draw[shift={( 7.5,0)}, color=black!30!white] (0,-0.25)--(1.5,-0.25)--(1.5,2.25)--(0,2.25)--cycle;

\fill[shift={( 9.0,0)}, color=black!15!white] (0,1.0)--(0,1.5) to[out=0,in=180] (1.5,1.5)--(1.5,1.0) to[out=180,in=0] (0,1.0)--cycle; 
\draw[shift={( 9.0,0)}] (0,0.0) to[out=0,in=90] (0.75,-0.25) to[out=90,in=180] (1.5,0.5);
\draw[shift={( 9.0,0)}] (0,0.5) to[out=0,in=180]                              (1.5,0.0);
\draw[shift={( 9.0,0)}] (0,1.0) to[out=0,in=180]                              (1.5,1.0);
\draw[shift={( 9.0,0)}] (0,1.5) to[out=0,in=180]                              (1.5,1.5);
\draw[shift={( 9.0,0)}] (0,2.0) to[out=0,in=180]                              (1.5,2.0);
\draw[shift={( 9.0,0)}, color=black!30!white] (0,-0.25)--(1.5,-0.25)--(1.5,2.25)--(0,2.25)--cycle;

\fill[shift={(10.5,0)}, color=black!15!white] (0,1.0)--(0,1.5) to[out=0,in=180] (1.5,1.0)--(1.5,0.5) to[out=180,in=0] (0,1.0)--cycle; 
\draw[shift={(10.5,0)}] (0,0.0) to[out=0,in=90] (0.75,-0.25) to[out=90,in=180] (1.5,2.0);
\draw[shift={(10.5,0)}] (0,0.5) to[out=0,in=180]                              (1.5,0.0);
\draw[shift={(10.5,0)}] (0,1.0) to[out=0,in=180]                              (1.5,0.5);
\draw[shift={(10.5,0)}] (0,1.5) to[out=0,in=180]                              (1.5,1.0);
\draw[shift={(10.5,0)}] (0,2.0) to[out=0,in=180]                              (1.5,1.5);
\draw[shift={(10.5,0)}, color=black!30!white] (0,-0.25)--(1.5,-0.25)--(1.5,2.25)--(0,2.25)--cycle;

\node at (-0.25,0.0) {\small 1};
\node at (-0.25,0.5) {\small 2};
\node at (-0.25,1.0) {\small 3};
\node at (-0.25,1.5) {\small 4};
\node at (-0.25,2.0) {\small 5};
\node at (12.25,0.0) {\small 1};
\node at (12.25,0.5) {\small 2};
\node at (12.25,1.0) {\small 3};
\node at (12.25,1.5) {\small 4};
\node at (12.25,2.0) {\small 5};

\node[shift={( 0.0,0)}] at (0.75,2.55) {$C_3$};
\node[shift={( 1.5,0)}] at (0.75,2.55) {$C_5$};
\node[shift={( 3.0,0)}] at (0.75,2.55) {$C_1$};
\node[shift={( 4.5,0)}] at (0.75,2.55) {$C_5$};
\node[shift={( 6.0,0)}] at (0.75,2.55) {$C_2$};
\node[shift={( 7.5,0)}] at (0.75,2.55) {$C_5$};
\node[shift={( 9.0,0)}] at (0.75,2.55) {$C_2$};
\node[shift={(10.5,0)}] at (0.75,2.55) {$C_5$};
\end{tikzpicture}
}

\def\picFA{\begin{tikzpicture}[thick,scale=0.66]
\draw[color=black!30!white] (-0,0)--(16,0);
\draw[color=black!30!white] (-0,1)--(16,1);
\draw[color=black!30!white] (-0,2)--(16,2);
\draw[color=black!30!white] (-0,3)--(16,3);
\draw[color=black!30!white] ( 0,-0)--( 0,3);
\draw[color=black!30!white] ( 1,-0)--( 1,3);
\draw[color=black!30!white] ( 2,-0)--( 2,3);
\draw[color=black!30!white] ( 3,-0)--( 3,3);
\draw[color=black!30!white] ( 4,-0)--( 4,3);
\draw[color=black!30!white] ( 5,-0)--( 5,3);
\draw[color=black!30!white] ( 6,-0)--( 6,3);
\draw[color=black!30!white] ( 7,-0)--( 7,3);
\draw[color=black!30!white] ( 8,-0)--( 8,3);
\draw[color=black!30!white] ( 9,-0)--( 9,3);
\draw[color=black!30!white] (10,-0)--(10,3);
\draw[color=black!30!white] (11,-0)--(11,3);
\draw[color=black!30!white] (12,-0)--(12,3);
\draw[color=black!30!white] (13,-0)--(13,3);
\draw[color=black!30!white] (14,-0)--(14,3);
\draw[color=black!30!white] (15,-0)--(15,3);
\draw[color=black!30!white] (16,-0)--(16,3);
\draw[ultra thick] (0,0)--(3,3)--(4,2)--(5,3)--(8,0)--(10,2)--(12,0)--(14,2)--(16,0); 
\node at (0,-0.5) {\small $(0,0)$};
\node at (16,-0.5) {\small $(16,0)$};
\end{tikzpicture}}

\newtheorem{theorem}{Theorem}
\newtheorem{corollary}[theorem]{Corollary}
\newtheorem{lemma}[theorem]{Lemma}
\newtheorem{proposition}[theorem]{Proposition}

\theoremstyle{definition}
\newtheorem{definition}{Definition}

\newtheorem*{claim}{Claim}

\allowdisplaybreaks

\title{Juggling card sequences}
\author{Steve Butler\thanks{Dept.\ of Mathematics, Iowa State University, Ames, IA 50011, USA {\tt butler@iastate.edu}.  Partially supported by an NSA Young Investigator grant.} \and 
Fan Chung\thanks{Dept.\ of Mathematics and Dept.\ of Computer Science and Engineering, UC San Diego, La Jolla, CA 92093, USA {\tt fan@ucsd.edu}.} \and 
Jay Cummings\thanks{Dept.\ of Mathematics, UC San Diego, La Jolla, CA 92093, USA {\tt jjcummings@ucsd.edu}.} \and 
Ron Graham\thanks{Dept.\ of Computer Science and Engineering, UC San Diego, La Jolla, CA 92093, USA {\tt graham@ucsd.edu}.}}

\date{April 6, 2015}

\begin{document}
\maketitle

\begin{abstract}
Juggling patterns can be described by a sequence of cards which keep track of the relative order of the balls at each step.  This interpretation has many algebraic and combinatorial properties, with connections to Stirling numbers, Dyck paths, Narayana numbers, boson normal ordering, arc-labeled digraphs, and more.  Some of these connections are investigated with a particular focus on enumerating juggling patterns satisfying certain ordering constraints, including where the number of crossings is fixed.
\end{abstract}

\section{Introduction}
It is traditional for mathematically-inclined jugglers to represent various juggling patterns by sequences $T=(t_1, t_2, \ldots, t_n)$ where the $t_i$ are natural numbers. The connection to juggling being that at time $i$, the object (which we will assume is a \emph{ball}) is thrown so that it comes down $t_i$ time units later at time $i+t_i$.  The usual convention is that the sequence $T$ is repeated indefinitely, i.e., it is periodic, so that the expanded pattern is actually $(\ldots,t_1, t_2, \ldots, t_n, t_1, t_2, \ldots, t_n, \ldots)$.

A sequence $T$ is said to be a \emph{juggling sequence}, or \emph{siteswap sequence}, provided that it never happens that two balls come down at the same time. For example, $(3,4,5)$ is a juggling sequence while $(3,5,4)$ is not. It is known \cite{BEGW} that a necessary and sufficient condition for $T$ to be a juggling sequence is that all the quantities $i + t_i \pmod n$ are distinct. For a juggling sequence $T=(t_1, t_2, \ldots, t_n)$, its \emph{period} is defined to be $n$. A well known property is that the number of balls $b$ needed to perform $T$ is the average $b = \frac1n\sum_{i=1}^n t_i$. It is also known that the number of juggling sequences with period $n$ and at most $b$ balls is $b^n$ (cf.\ \cite{BG,BEGW}; our convention assumes that we will always catch and then immediatly throw something at every step, or in other words there are no $0$ throws).

There is an alternative way to represent periodic juggling patterns, a variation of which was first introduced by Ehrenborg and Readdy \cite{ER}. For this method, certain \emph{cards} are used to indicate the relative ordering of the balls (with respect to when they will land) as the juggling pattern is executed. One might call the first representation ``time'' sequences for representing juggling patterns while the second representation might be called ``order'' sequences for representing these same patterns. 

In this paper, we will explore various algebraic and combinatorial properties associated with these order  sequences. It will turn out that there are a number of unexpected connections with a wide variety of combinatorial structures.  In the remainder of this section we will introduce these juggling card sequences, and then in the ensuing sections will count the number of juggling card sequences that induce a given ordering, count the number of juggling card sequences that do not change the ordering and have a fixed number of crossings, and look at the probability that the induced ordering consists of a single cycle.

\subsection{Juggling card sequences}
We will represent juggling patterns by the use of \emph{juggling cards}. Sequences of these juggling cards will describe the behaviors of the balls being juggled.  In particular, the set of juggling cards produce the juggling diagram of the pattern.

Throughout the paper, we will let $b$ denote the number of balls that are available to be juggled. We will also have available to us a collection of cards $\mathcal{C}$ that can be used.  In the setting when at each time step one ball is caught and then immediately thrown, we can represent these by $C_1,C_2,\ldots,C_b$ where $C_i$ indicates that the bottom ball in the ordering has now dropped into our hand and we now throw it so that relative to the other balls it will now be the $i$-th ball to land.  Visually we draw the cards so that there are $b$ levels on each side of the card (numbered \emph1,\emph2,\ldots,\emph{b} from bottom to top) and $b$ tracks connecting the levels on the left to the levels on the right by the following:  level \emph1 connects to level $i$; level $j$ connects to level $j-\emph1$ for $2\le j\le i$; level $j$ connects to level $j$ for $i+1\le j\le b$.  An example of the cards when $b=4$ is shown in Figure~\ref{4-cards}.

\begin{figure}[hftb]
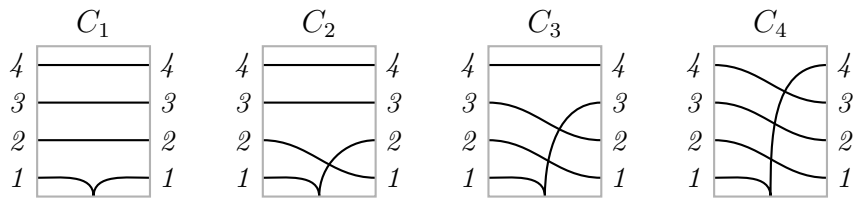

    \centering
    \picAA
    \caption{Cards for $b=4$}
    \label{4-cards}
\end{figure}

As we juggle the $b$ balls, $1,2,\ldots,b$, move along the track on the cards.  For each card $C_i$ the relative ordering of the balls changes and corresponds to a permutation $\pi_{C_i}$.  Written in cycle form this permutation is $\pi_{C_i}=(i~i{-}1~\ldots~2~1)$.  In particular, a ball starting on level $j$ on the left of card $C_i$ will be on level $\pi_{C_i}(j)$ on the right of card $C_i$.

A sequence of cards, $A$, written by concatenation, i.e., $C_{i_1}C_{i_2}\ldots C_{i_n}$, is a \emph{juggling card sequence} of length $n$.  The $n$ cards of $A$ are laid out in order so that the levels match up.  The balls now move from the left of the sequence of cards to the right of the sequence of cards with their relative ordering changing as they move.  The resulting final change in the ordering of the balls is a permutation denoted  $\pi_A$, i.e., a ball starting on level $i$ will end on $\pi_A(i)$.  We note that $\pi_A=\pi_{C_{i_1}} \pi_{C_{i_2}} \cdots \pi_{C_{i_n}}$.  We will also associate with juggling card sequence $A$ the arrangement $[\pi_A^{-1}(1), \pi_A^{-1}(2), \ldots, \pi_A^{-1}(b)]$, which corresponds to the resulting ordering of the balls on the right of the diagram when read from bottom to top.

As an example, in Figure \ref{4-card_sequence} we look at $A=C_3C_3C_2C_4C_3C_4C_3C_2C_2$ (note we allow ourselves the ability to repeat cards as often as desired).  For this juggling card sequence we have $\pi_{A} = (1~2~4~3)$ and corresponding arrangement $[3,1,4,2]$.  We have also marked the ball being thrown at each stage under the card for reference.

\begin{figure}[hftb]
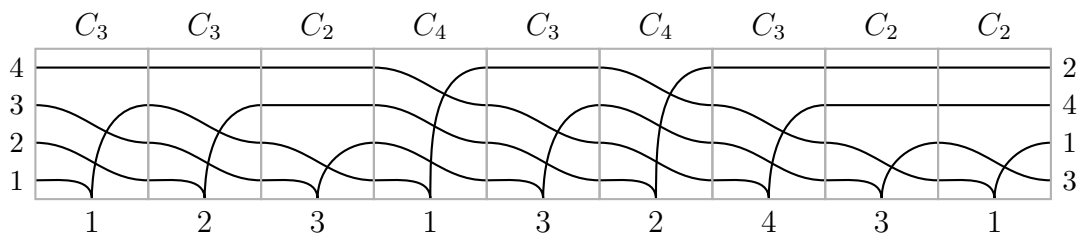

    \centering
    \picAC
    \caption{A juggling card sequence $A$; below each card we mark the ball thrown}
    \label{4-card_sequence}
\end{figure}

From the juggling card sequence we can recover the siteswap sequence by letting $t_i$ be the number of cards traversed starting at the bottom of the $i$th card until we return to the bottom of some other card.  For example, the siteswap pattern in Figure~\ref{4-card_sequence} is $(3,4,2,5,3,10,5,2,2)$.

We can also increase the number of balls caught and then thrown at one time, which is known as \emph{multiplex} juggling.  In the more general setting we will denote the cards $C_S$ where $S=(s_1,s_2,\ldots,s_k)$ is an ordered subset of $[b]$.  Each card still has levels $\emph1,\emph2,\ldots,b$ and now for $1\le j \le k$ the ball at level $i$ goes to level $s_i$ and the remaining balls then fill the available levels in a way that preserves their order.  As an example, the cards $C_{2,5}$ and $C_{5,2}$ are shown in Figure~\ref{25_cards} for $b=5$.

\begin{figure}[hftb]
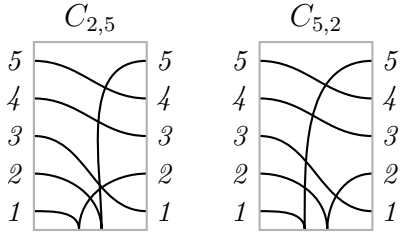

    \centering
    \picAD
    \caption{The cards $C_{2,5}$ and $C_{5,2}$ for $b=5$}
    \label{25_cards}
\end{figure}

As before we can combine these together to form juggling card sequences $A$ which induce permutations $\pi_A$ and corresponding arrangements.  An example of a juggling card sequence composed of cards $C_S$ with $|S|=2$ is shown in Figure~\ref{fig:juggexample} which has corresponding arrangement $[3,4,2,5,1]$.  We note that it is also possible to form juggling card sequences which have differing sizes of $|S|$.

\begin{figure}[hftb]
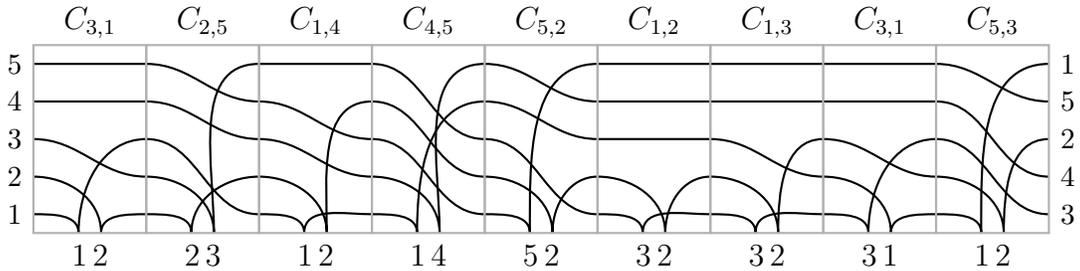

    \centering
    \picZZ
    \caption{A juggling card sequence $A$; below each card we mark the balls thrown}
    \label{fig:juggexample}
\end{figure}

\section{Juggling card sequences with given arrangement}\label{sec3}
In this section we will consider the problem of enumerating juggling card sequences of length $n$ using cards drawn from a collection of cards $\mathcal{C}$ with the final arrangement corresponding to the permutation $\sigma$.  We will denote the number of such sequences by $JS(\sigma, n, \mathcal{C})$.  This will be dependent on the following parameter.

\begin{definition}
Let $\sigma$ be a permutation of $1,2,\ldots,b$.  Then $L(\sigma)$ is the largest $\ell$ such that $\sigma(b-\ell+1)<\cdots<\sigma(b-1)<\sigma(b)$.  Alternatively, $L(\sigma)$ is the largest $\ell$ so that $b-\ell+1,\ldots,b-1,b$ appear in increasing order in the arrangement for $\sigma$.
\end{definition}

As an example, the final arrangement in Figure~\ref{4-card_sequence} has $L(\sigma)=2$ and the final arrangement in Figure~\ref{fig:juggexample} has $L(\sigma)=3$.

The key idea for our approach will be that with information about what balls are thrown we can ``work backwards''.  In particular, we have the following.

\begin{proposition}\label{prop:backone}
Given a single card, if we know the ordering of balls on the right hand side of the card \emph{and} we know which balls are thrown, then we can determine the card $C_S$ and the ordering of the balls on the left hand side of the card.
\end{proposition}
\begin{proof}
Suppose that $i_1,i_2,\ldots,i_\ell$ are the balls, in that order, which are thrown.  Then the card is $C_S$ where $S=(s_1,s_2,\ldots,s_\ell)$ and $s_j$ is the location of ball $i_j$ in the ordering of the balls (i.e., where the ball $i_j$ moved).  The ordering of the left hand side starts $i_1,i_2,\ldots,i_\ell$ and the remaining balls are then determined by noting that their ordering must be preserved.
\end{proof}

\subsection{Throwing one ball at a time}
We now work through the case when one ball at a time is caught and then immediately thrown.

\begin{theorem}\label{fact1a}
Let $b$ be the number of balls and $\mathcal{C}=\{C_1,\ldots,C_b\}$.  Then 
\[
    JS(\sigma, n, \mathcal{C}) = \sum_{k=b-L(\sigma)}^b \bigg\{\!{n\atop k}\!\bigg\}, 
\]
where $\big\{\!{n\atop k}\!\big\}$ denotes the Stirling numbers of the second kind.
\end{theorem}

\begin{proof}
We establish a bijection between the partitions of $[n]$ into $k$ nonempty subsets $[n]=X_1 \cup X_2 \cup \cdots \cup  X_k$ where $b-L(\sigma)\le k\le b$ and juggling card sequences of length $n$ using cards from $\mathcal{C}$ with the final arrangement corresponding to $\sigma$.  Because such partitions are counted by the Stirling numbers of the second kind, the result will then follow.

Starting with a partition we first reindex the sets so that the minimal elements are in increasing order, i.e., $\min X_i<\min X_j$ for $i<j$.  We now place $n$ \emph{blank} cards, mark the final arrangement corresponding to $\sigma$ on the right of the final card, and then under the $i$-th card we write $j$ if and only if $i\in X_j$.

We interpret the labeling under the cards as the ball that is thrown at that card, in particular we will have that $k$ of the balls are thrown.  We can now apply Proposition~\ref{prop:backone} iteratively from the right hand side to the left hand side to determine the cards in the juggling card sequence, where we update our ordering as we move from right to left.

We claim that the final ordering that we will end up with on the left hand side is $[1,2,\ldots,b]$ so that this is a juggling card sequence which should be counted.  Looking at the proof of the proposition we see that at each step the only ball which changes position in the ordering is the ball which is thrown, and in that case the ball was thrown from the bottom of the ordering.  We now have two observations to make:
\begin{itemize}
\item For the $k$ balls that will be thrown they will move into the first $k$ slots in the ordering, and by the assumption of our indexing we have that the first $k$ balls are ordered, i.e., for $1\le i<j\le k$ the first occurrence when going from left to right of $i$ is before the first occurrence of $j$ so that $i$ will move below $j$.
\item The remaining balls will not have their relative ordering change.  However, by our assumption on $k$ we have that $k+1,\ldots,b$ are already in the proper ordering.
\end{itemize}

This establishes the map from partitions to juggling card sequences.  To go in the other direction, we take a juggling card sequence of length $n$ using our cards from $\mathcal{C}$, write down which ball is thrown under each card, and then form our sets for the partition by letting $X_i$ be the location of the cards where ball $i$ is thrown.  Because $\sigma(b-L(\sigma))>\sigma(b-L(\sigma)+1)$ it must be that at some time that the ball $b-L(\sigma)$ was thrown and therefore the number of sets in our partition is at least $b-L(\sigma)$.  This finishes the other side of the bijection and the proof.
\end{proof}

For the partition of $[9] = \{1,4,9\} \cup \{2,6\} \cup \{3,5,8\} \cup \{7\}$ with final arrangement $[3,1,4,2]$ the juggling card sequence which will be formed is the one given in Figure~\ref{4-card_sequence}.

\subsection{Throwing $m\ge 2$ balls at a time}
The proof readily generalizes to the setting where we catch and then immediately throw $m$ balls at a time.  What we need to do is to find the appropriate way to generalize the Stirling numbers of the second kind.

\begin{definition}
Given $n$ and $k$ let $X=\{x_1,x_2,\ldots,x_k\}$.  Then $\big\{\!{n\atop k}\!\big\}_{\!m}$ is the number of ways, up to relabeling the $x_i$, to form $Y_1,Y_2,\ldots,Y_n$ so that $Y_j=(x_{j_1},\ldots,x_{j_m})$ is an ordered subset of $X$ and each $x_i$ is in at least one $Y_j$.
\end{definition}

We note that $\big\{\!{n\atop k}\!\big\}_{\!1}=\big\{\!{n\atop k}\!\big\}$.  This can be seen by observing that each $Y_i$ is a single entry and then we form our partition by grouping the indices of the $Y_i$ which agree.  We now show that this gives the appropriate generalization.

\begin{theorem}
Let $b$ be the number of balls and $\mathcal{C}$ be the collection of all cards for which $m$ balls are thrown.  Then 
\[
    JS(\sigma, n, \mathcal{C}) = \sum_{k=b-L(\sigma)}^b \bigg\{\!{n\atop k}\!\bigg\}_{\!m}. 
\]
\end{theorem}
\begin{proof}
Suppose we are given $Y_1,\ldots,Y_n$ with $Y_j=(x_{j_1},\ldots,x_{j_m})$ an ordered subset of $\{x_1,\ldots,x_k\}$.  Then we first concatenate the $Y_j$ together and remove all but the \emph{first} occurrence of each $x_i$ leaving us with a list $Y'$.  By our assumptions we have that $Y'$ consists of $x_1,\ldots,x_k$ in some order.  For $Y_1,\ldots, Y_n$, we now replace $x_1,\ldots,x_k$ by $1,\ldots,k$ by replacing $x_i$ with $j$ if $x_i$ is in the $j$-th position of $Y'$.  (This process is equivalent to the reindexing carried out in the special case when one ball is thrown at a time.)

We now have $Y_1,\ldots,Y_n$ with each consisting of $m$ distinct numbers drawn from $\{1,\ldots,k\}$ with the property that if $i<j$ then $i$ appears before $j$ (i.e., in the sense that if the first occurrence of $i$ is in $Y_p$ and the first occurrence of $j$ is in $Y_q$ and then either $p<q$ or $p=q$ and $i$ appears in the list before $j$ in $Y_p$).  We now put down $n$ blank cards, write down the arrangement corresponding to $\sigma$ on the right side of the last card and write $Y_i$ under the $i$th card for all $i$.  The remainder of the proof then proceeds as before, i.e., we can now work from right to left and determine the card used at each stage by Proposition~\ref{prop:backone}.  The resulting process gives a valid juggling sequence because the initial arrangement will have the first $k$ balls in order (by our work on reindexing) and the final balls inherit their order, which by assumption were already in the correct order.

The map in the other direction is carried out as before, i.e., given a juggling card sequence under each card we write the balls which are thrown and use these to form $Y_1,\ldots,Y_n$ which contribute to the count of $\big\{\!{n\atop k}\!\big\}_{\!m}$ for some appropriate $k$.
\end{proof}

The value $\big\{\!{n\atop k}\!\big\}_{\!2}$ is found by counting sets of ordered pairs.  In particular, this counts the number of multi-digraphs with $n$ \emph{labeled} edges and $k$ vertices.  This leads to a bijection between these digraphs and juggling sequences for a given $\sigma$, provided $k\ge b-L(\sigma)$.  As an example consider the edge-labeled directed graph shown in Figure~\ref{digraph}.

\begin{figure}[htbf]
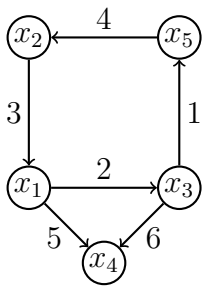

    \centering
    \picAF
    \caption{An edge labeled multi-digraph}
    \label{digraph}
\end{figure}

Using the edge labeling we can now form the sets so that $Y_1=(x_3,x_5)$, $Y_2=(x_1,x_3)$, $Y_3=(x_2,x_1)$, $Y_4=(x_5,x_2)$, $Y_5=(x_1,x_4)$ and $Y_6=(x_3,x_4)$.  We now need to label the $x_i$ with $1,2,3,4,5$ so that the first occurrences of each number (ball) is increasing.  To do this we first concatenate these lists together to form the following (i.e., the occurrences in order of all of the $x_i$):
\[
(x_3,x_5, x_1,x_3, x_2,x_1, x_5,x_2, x_1,x_4, x_3,x_4)
\]
From here we look at first occurrences of each $x_i$ which is found by removing all but the first occurrence of each symbol which gives us the following list.
\[
Y' = (x_3,x_5,x_1,x_2,x_4)
\]
Therefore to make sure we have the first occurrences in the proper order, we replace $x_3,x_5,x_1,x_2,x_4$ by $1,2,3,4,5$ respectively.  If we now set the final arrangement to be $[4,5,2,1,3]$ then we get the corresponding juggling card sequence shown in Figure~\ref{sequence_2}.

\begin{figure}[hftb]
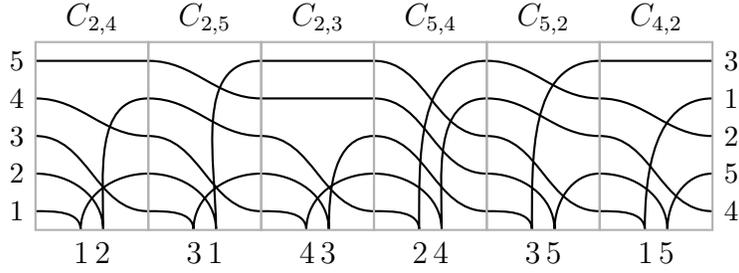

    \centering
    \picAE
    \caption{The juggling card sequence corresponding to the digraph from Figure~\ref{digraph} and final arrangement $[4,5,2,1,3]$}
    \label{sequence_2}
\end{figure}

This bijection gives us the following.

\begin{theorem}\label{thm:digraph}
Let $\sigma$ be a permutation of $1,2,\ldots,b$ and $b-L(\sigma)\le k\le b$.  Then there is a bijection between edge-labeled (multi-)digraphs without loops which have $n$ arcs on $k$ vertices and juggling card sequences $A$ of length $n$ where two balls are caught and thrown at a time, a total of $k$ balls are thrown, and satisfying $\pi_A=\sigma$.
\end{theorem}

We note that the numbers $\big\{\!{n\atop k}\!\big\}_{\!m}$ have appeared recently in the literature in connection with the so-called \emph{boson normal ordering problem} arising in statistical physics \cite{boson_1,boson_2}. The sequence $\big\{\!{n\atop k}\!\big\}_{\!2}$ is {\tt A078739} in the OEIS \cite{oeis}. 

For general $m$ it has been observed \cite{hell} that $\big\{\!{n\atop k}\!\big\}_{\!m}$ is the number of ways to properly color the graph $nK_m$ using exactly $k$ colors, i.e., each $Y_i$ is the coloring on the $i$-th copy of $K_m$, and by definition all $k$ colors must be used.  

If we denote the \emph{falling factorial} $x^{\underline{m}} = x(x-1)(x-2) \cdots (x-m+1)$, then the ordinary Stirling numbers $\big\{\!{n\atop k}\!\big\}$ act as \emph{connection coefficients} between $x^{\underline{n}}$ and $x^n$ by means of the formula (e.g., see \cite{gkp})
\[
    x^n = \sum_{k=1}^n \bigg\{\!{n\atop k}\!\bigg\} x^{\underline{k}}.
\]
In particular, they satisfy the recurrence:
\[
    \bigg\{\!{n+1\atop k}\!\bigg\}= k\bigg\{\!{n\atop k}\!\bigg\} + \bigg\{\!{n\atop k-1}\!\bigg\},
\]
and have the explicit representation
\[
    \bigg\{\!{n\atop k}\!\bigg\} = \frac{(-1)^k}{k!} \sum_{i=1}^k (-1)^i \binom{k}{i} i^n.
\]
The $\big\{\!{n\atop k}\!\big\}_{\!m}$ satisfy analogs of these three relationship.  Namely, as connection coefficients
\begin{equation*}
   (x^{\underline{m}})^n = \sum_{k=m}^{mn} \bigg\{\!{n\atop k}\!\bigg\}_{\!m} x^{\underline{k}},
\end{equation*}
satisfying a recurrence
\begin{gather*}
   \bigg\{\!{n+1\atop k}\!\bigg\}_{\!m} = \sum_{i=0}^m \binom{k+i-m}{i} \,m^{\underline{i}} \,\bigg\{\!{n\atop k+i-m}\!\bigg\}_{\!m},
\end{gather*}
and with the explicit representation
\begin{equation*}
   \bigg\{\!{n\atop k}\!\bigg\}_{\!m} = \frac{(-1)^k}{k!} \sum_{i=m}^k (-1)^i \binom{k}{i}(i^{\underline{m}})^n.
\end{equation*}

\subsection{Throwing different numbers of balls at different times} 
We have restricted our analysis to the case when our collection of cards all catch and then throw the same number of balls.  We can relax this restriction and allow ourselves to catch and throw differing number of balls at each step.  For example, we could insist that at the $i$-th step that $m_i$ balls are thrown.

Under the card in the $i$-th position we place a sequence $Y_i= (y_{i,1}, y_{i,2}, \ldots, y_{i,m_i})$.  We then concatenate the labels as before to give a mapping from the $y_{i,j}$ to $[k]$ to give a compatible ball  assignment to the card positions.  Then we work from right to left and recover the unique juggling card sequence which corresponds to this collection of ordered sets. A variation of the preceding arguments show that the number of such card sequences is equal to the number of $k$-colorings of $\cup_{i=1}^n K_{m_i}$.

A much more complete analysis of this problem with connections to generalized Stirling numbers and the boson normal ordering problem appears in \cite{EGH}. A good survey of this general problem also can be found in of \cite[Ch.\ 10]{Man}.

\section{Preserving ordering while throwing}
In the preceding section when we threw multiple balls at one time, we did not worry about preserving the ordering of the balls which were thrown.  The goal of this section is to add the extra condition that the relative order of the thrown balls is preserved, e.g., for $m=2$ our set of cards will be the set of ${b\choose2}$ cards given by $\{C_{i,j}:1\le i<j\le b\}$.  We will see that this situation is more complicated than the one in the preceding section.

To begin the analysis, we start with a $2$-cover of the set $[n]$. This is a collection of $k$ (not necessarily distinct) subsets $S_i$  of $[n]$ with the property that each element $j$ of $[n]$ occurs in exactly two of the $S_i$. We can represent a $2$-cover by a $k \times n$ matrix $M$ where for $1 \leq i \leq k$, $1 \leq j \leq n$, we have $M(i,j) = 1$ if $j \in S_i$, and $M(i,j) = 0$ otherwise.  For each set $S_i$ we will associate a \emph{virtual ball} $x_i$. For $1 \leq j \leq n$, we define the 2-element set $B_j = \{x_i : j \in S_i\}$. In other words, $x_i \in B_j$ if and only if $M(i,j) = 1$. The interpretation is that at time $j$, the two virtual balls $x_i \in B_j$ will be the balls that are thrown at that time.

We now produce the (unique) mapping  between the actual balls and the virtual balls $x_i$. To do this, we define a partial order on the $x_i$ as follows: $x_u$ is less than $x_v$, written as $x_u \prec x_v$, if among all the $B_i \neq \{x_u, x_v\}$, $x_u$ occurs \emph{before} $x_v$ (i.e., with a lower indexed  $B_i$). If there are no such $B_i$, we say that $x_u$ and $x_v$ are equivalent.

As an example, a $2$-cover of $[7]$ with five subsets is given by the following matrix.
\[
M=\bordermatrix{
~&B_1&B_2&B_3&B_4&B_5&B_6&B_7 \cr
x_1&1&0&1&0&0&0&1 \cr
x_2&0&1&0&0&1&0&0 \cr
x_3&1&0&0&1&0&1&0 \cr
x_4&0&1&0&0&1&0&0 \cr
x_5&0&0&1&1&0&1&1}
\]
We have labeled the rows of $M$ with the $x_i$ and the columns with the $B_j$. Thus, we see that  $x_2$ and $x_4$ are equivalent, so that the partial order on the $x_i$ is 
\begin{equation*}
    x_1 \prec x_3 \prec {x_2 \equiv x_4} \prec x_5.
\end{equation*}

If in the current arrangement we have that $u$ is below $v$, then $v$ cannot be thrown before $u$ (though it might possibly be at the same time).  Therefore the partial order on the $x_i$ determines how the balls are positioned relative to one another. The partial order doesn't specify anything about the relative order of equivalent $x_i$ but because such pairs are always thrown together, their relative order never changes during the process of traversing all the cards in the sequence.

In Figure~\ref{fig5} we show the sequence generated by the $2$-cover from $M$,  where we assume the finishing arrangement of the $x_i$ is from bottom to top $x_4,x_1,x_5,x_3,x_2$. This choice was arbitrary, except that the initial and terminal orders of the equivalent pair $x_2$ and $x_4$ must be the same, since there is a unique initial sequence which can have the $x_i$ in $B_j$ being thrown at time $j$, namely, the sequence that is consistent with the partial order $\prec$ on the $x_i$. To determine the appropriate cards needed for the required throwing patterns it is simply a matter of starting at the right hand side and choosing the cards sequentially which achieve the required throws. In Figure~\ref{fig5}, we have also have indicated the corresponding cards $C_{i,j}$ which accomplish the indicated throws.

\begin{figure}[htbf]
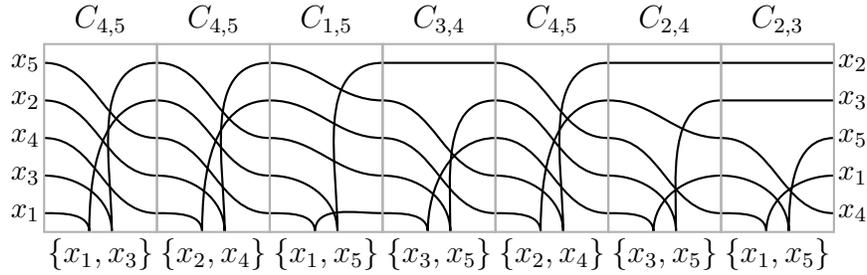

    \centering
    \picAG
    \caption{A card sequence for the matrix $M$}
    \label{fig5}
\end{figure}
 
If we now make the identification $x_1\rightarrow1, x_3 \rightarrow2, x_4 \rightarrow 3,  x_2 \rightarrow 4, x_5 \rightarrow 5$, then we have the picture shown in Figure~\ref{fig6}.

\begin{figure}[ht]
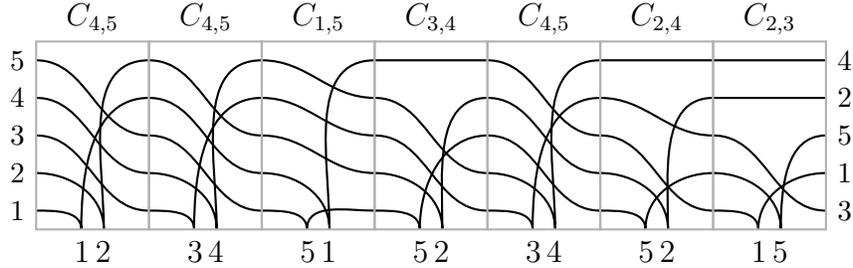

    \centering
    \picAH
    \caption{A card sequence for the matrix $M$ using actual balls}
    \label{fig6}
\end{figure}

We can achieve any permutation $\sigma$ of the balls $\{1,2,3,4,5\}$ starting in increasing order provided only that $\sigma(2)$ is below $\sigma(4)$. 

For general $n$ and $k$, given a $2$-cover of $[n]$ with $k$ sets $S_1, \ldots, S_k$, there is an induced partial order on the sets (or what we called virtual balls). For any terminal permutation $\sigma$ which preserves the relative order of equivalent balls, there is a unique sequence of cards which achieves this permutation. 

As pointed out in \cite{cps}, there is a direct correspondence between 2-covers of $[n]$ with $k$ subsets and multigraphs $G(n,k)$ having $k$ vertices and $n$ labeled edges.  In the case of graphs, the vertices of $G$ will be $\{x_1, x_2, \ldots, x_k\}$. We insert the edge $\{x_r, x_s\}$ with label $i$ if the $i$-th column of  $M$ has $1$'s in rows $r$ and $s$. The number of vertices of such an edge-labeled multigraph corresponds to the number of balls which are thrown. These are enumerated by the numbers of vertices and labeled edges in \cite{labele} (see also {\tt A098233} in the OEIS \cite{oeis}). We illustrate this connection in Figure~\ref{multigraphs} where we show the three edge-labeled multigraphs on two edges and the corresponding card sequences which generate the identity permutation.

\begin{figure}[htp]
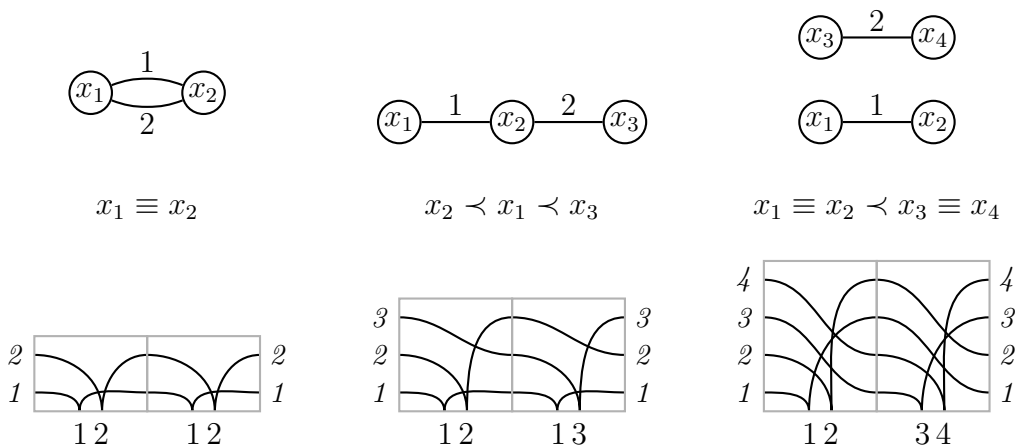

    \centering
\begin{tabular}{c@{\qquad}c@{\qquad}c}
\picBA & \picBB & \picBC \\[12pt]
$x_1\equiv x_2$ & $x_2\prec x_1 \prec x_3$ & $x_1\equiv x_2 \prec x_3 \equiv x_4$ \\[12pt]
\picBD & \picBE & \picBF
\end{tabular}
    \caption{Edge-labeled multigraphs with two edges, and the corresponding card sequences}
    \label{multigraphs}
\end{figure}

In the special case that the desired permutation $\pi_A =\sigma = \text{id}$, the identity permutation, then \emph{any} $2$-cover can generate this permutation.  This gives the following result (which should be compared with Theorem~\ref{thm:digraph}).

\begin{theorem}\label{thm:graph}
Let $b$ be the number of balls.  Then there is a bijection between edge-labeled (multi-)graphs without loops which have $n$ edges on $b$ vertices and juggling card sequences $A$ of length $n$ where two balls are caught and thrown at a time and the relative ordering of the thrown balls is preserved, where all $b$ of the balls are thrown, and satisfying $\pi_A=\text{id}$.
\end{theorem}

The asymptotic behavior of the number of $2$-covers of an $n$-set, denoted $Cov(n)$,  has been studied in \cite{cps}. In particular, it is shown there that
\[
Cov(n) \sim B_{2n}2^{-n}\exp \left(-\frac{1}{2} \log\left(\frac{2n}{\log n}\right)\right)
\]
where $B_{2n}$ is the well-known Bell number (see \cite{amo}).

Counting the number of juggling card sequences which generate permutations other than the identity is more complicated.

In the more general case of throwing $m \geq 3$ balls, we want to consider  $m$-covers of the set $[n]$. An $m$-cover of $[n]$ is a collection of $k$ (not necessarily distinct) subsets $S_i$  of $[n]$ with the property that each element $j$ of $[n]$ occurs in exactly $m$ of the $S_i$. As before, we can represent the $m$-cover by a $k \times n$ matrix $M$ where for $1 \leq i \leq k$, $1 \leq j \leq n$, $M(i,j) = 1$ if $j \in S_i$, and $M(i,j) = 0$ otherwise.

The same analysis holds in this case of general $m$ as in the case of $m=2$. Namely, for each subset $S_i$ in the $m$-cover, we can associate a virtual ball $x_i$. Then we can use the sets $B_j$ corresponding to the columns of $M$ to induce a partial order $\prec$ on the $x_i$. As before, any permutation $\sigma$ on $[k]$ which respects the order of equivalent elements can be achieved by a unique sequence of cards. In the case that $\sigma$ is the identity permutation, then any $m$-cover of $[n]$ is able to generate this permutation with an appropriate sequence of cards.  In this case the number of such juggling card sequences is the number of hyperedge-labeled multi-hypergraphs, (similar to the edge-labeled multigraphs for the case $m=2$).

\section{Juggling card sequences with minimal crossings}
We now return to throwing a single ball at a time.  Any juggling card sequence of $n$ cards will produce a valid siteswap sequence which has period $n$.  However most such siteswap sequences will result in having the balls be permuted amongst themselves after $n$ throws.  So one natural family to focus on are those which satisfy $\pi_A=\text{id}$, i.e., after $n$ throws the same balls are in the same position to repeat.

Suppose now we follow the balls as they traverse the cards of some sequence $A$.  Then when a card $C_k$ is used, we see that the path of the thrown ball has $k-1$ ``crossings'' in that card, i.e., locations where the tracks intersect.  For a sequence $A = C_{i_1} C_{i_2} \ldots C_{i_n}$, the total number of crossings is $Cr(A)=\sum (i_k-1)$.  In the case when a juggling card sequence has $b$ balls, uses the card $C_b$, and has $\pi_A = \text{id}$, then the number of crossings satisfies $Cr(A) \geq b(b-1)$.  To see this we note that \emph{every} ball must be thrown (i.e., we throw something up to track $b$ which moves $b$ down and so we must eventually have a throw that returns $b$ to the top).  In particular, the paths of each pair of balls $i$ and $j$, with $i \neq j$, must cross at least twice.

We will say a juggling card sequence $A$ is a \emph{minimal crossing juggling card sequence} if the sequence has $b$ balls, uses the card $C_b$, has $\pi_A=\text{id}$, and $Cr(A)=b(b-1)$.  The goal of this section is to count the number of minimal crossing juggling card sequences.  In the process we will give a structure result that can give a bijective relationship with Dyck paths.

\subsection{Bijection with Dyck paths}
Dyck paths are one of many well known combinatorial objects that are connected with the Catalan numbers.  Many of these objects can be decomposed into two smaller (possibly empty) objects with the same properties; and we start by showing that this is the case with minimal crossing juggling card sequences.

\begin{lemma}\label{lem:recurrence}
Given a minimal crossing juggling card sequence $A$ with $b$ balls using $n$ cards, there is a \emph{unique} pair of minimal crossing juggling card sequences $(B,C)$ so that $B$ uses $k$ balls, and $m$ cards and $C$ uses $b-k$ balls and $n-m-1$ cards (with the possibility that $B$ or $C$ might be empty).  Further, given any such pair of minimal crossing juggling card sequences $(B,C)$, the minimal crossing juggling card sequence $A$ can be determined.
\end{lemma}
\begin{proof}
The first card of $A$ will throw the ball up to some level $k+1$ and will thus cross paths with balls $2,\ldots,k+1$.  By the time that the first ball is thrown the second time, the first ball will have had to cross paths with balls $2,\ldots,k+1$ a second time.  Because each pair of balls can only cross twice it must be that the ball $1$ will never again cross with balls $2,\ldots,k+1$.  In particular, we will never throw balls $2,\ldots,k+1$ after we throw ball $1$ the second time.  From this we conclude that all the crossings between balls $2,\ldots,k+1$ will occur between the first two throws of ball $1$ and that the relative ordering of balls $2,\ldots,k+1$ will be set when we get to the second throw of ball $1$.

So between the first two throws of ball $1$, if we ignore balls $1,k+2,\ldots,b$ then we have a juggling card sequence for $k$ balls with $k(k-1)$ crossings with the final arrangement corresponding to the identity.

If we now ignore balls $2,\ldots,k+1$ from the second throw of ball $1$ until the end then we must again have all of the $(b-k)(b-k-1)$ crossings among the remaining balls with the final arrangement corresponding to the identity.

We can now conclude that every juggling card sequence that we want to count can be broken into the following three parts:
\begin{itemize}
\item The first card which throws ball $1$ to height $k+1$.
\item The set of cards between the first two occurrences of the throw of ball $1$; a juggling card sequence with $m$ cards and $k$ balls having $k(k-1)$ crossings and corresponding to the identity arrangement.  We denote this minimal crossing juggling card sequence by $B$.
\item The set of cards from the second time ball $1$ is thrown to the end; a juggling card sequence with $n-m-1$ cards and $b-k$ balls having $(b-k)(b-k-1)$ crossings and corresponding the identity arrangement.  We denote this minimal crossing juggling card sequence by $C$.
\end{itemize}

The first card can be found by knowing the number of balls used in $B$, so therefore we only need to know $B$ and $C$.  Further, given the above information, we can reconstruct the juggling card sequence for $A$.  Namely, we have the first card.  For the next set of cards as determined by $B$, we initially add balls $1,k+2,\ldots,b$ on top of the balls $2,\ldots,k+1$ and then we continue with the same cards as before \emph{except} for the last time each ball is thrown we increase the height of the throw to move above $1,k+2,\ldots,b$, i.e., the card $C_t$ will be replaced by $C_{t+b-k}$.  For the last set of cards as determined by $C$, we do the same process where we initially add balls $2,\ldots,k$ on the top and then we continue with the same cards as before \emph{except} for last time each ball $k+2,\ldots,b$ is thrown we increase the height of the throw to move above $2,\ldots,k+1$, i.e., the card $C_t$ will be replaced by $C_{t+k}$.
\end{proof}

To help illustrate the correspondence used in Lemma~\ref{lem:recurrence} in Figure~\ref{fig:recurin} we give two juggling card sequences with minimal crossings, one for $2$ balls and $3$ cards and the other for $3$ balls and $4$ cards.  In Figure~\ref{fig:recurout} we give the corresponding juggling card sequence; to help emphasize the structure we shade the portion of the balls which move in unison according to the construction in the lemma in the parts coming from $B$ and $C$.

\begin{figure}[hftb]
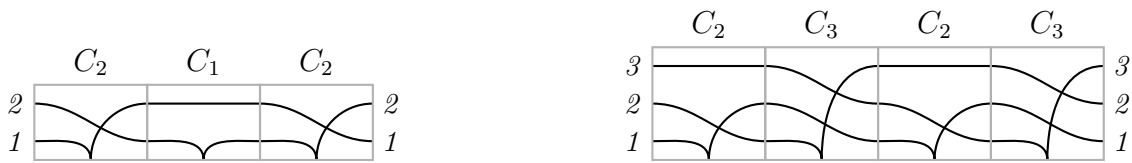

\centering
\picDA
\hfill
\picDB
\caption{Two minimal crossing juggling card sequences}
\label{fig:recurin}
\end{figure}

\begin{figure}[hftb]
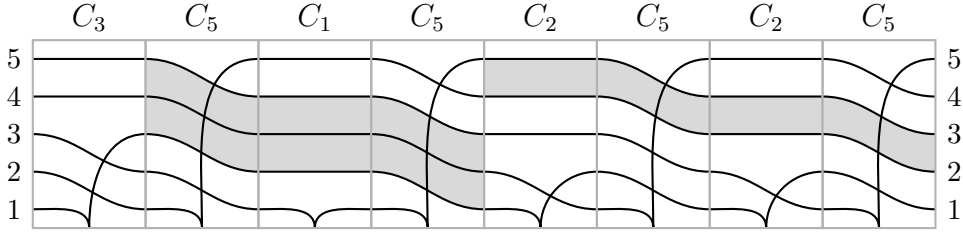

\centering
\picDC
\caption{The result of combining the two sequences in Figure~\ref{fig:recurin}}
\label{fig:recurout}
\end{figure}

Let us suppose that we indicate the preceding correspondence in the following way, if $B$ and $C$ are the minimal crossing juggling card sequences that generate the minimal crossing juggling card sequence $A$ then we write this as $A=(B)C$.  So that the example from Figures~\ref{fig:recurin} and \ref{fig:recurout} would be written as
\[
C_3C_5C_1C_5C_2C_5C_2C_5=(C_2C_1C_2)C_2C_3C_2C_3.
\]
Now we simply apply this convention recursively to each minimal crossing juggling card sequence, following the rule that if one part is empty we do not write anything.  So $(*)$ would be a juggling card sequence where ball $1$ does not return until the last card, $()*$ would be a juggling card sequence where the first card is $C_1$, and $()$ corresponds to the unique minimal juggling card sequence consisting of a single card, $C_1$.  If we now carry this out on the above example we get the following:
\begin{align*}
C_3C_5C_1C_5C_2C_5C_2C_5&=(C_2C_1C_2)C_2C_3C_2C_3\\
&=((C_1C_1))(C_1)C_2C_2\\
&=((()C_1))(())(C_1)\\
&=((()()))(())(())
\end{align*}
This leads naturally to Dyck paths by associating ``$($'' with an up and to the right step and ``$)$'' with a down and to the right step, which in our example gives the Dyck path shown in Figure~\ref{fig:Dyck}.  This process can be reversed (working from right to left and inside to outside), giving us a bijection between these minimal crossing juggling card sequences and Dyck paths.

\begin{figure}[htfb]
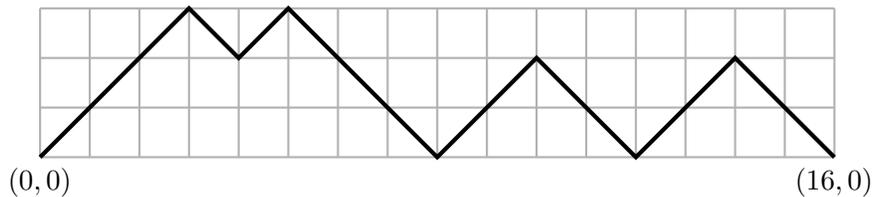

\centering
\picFA
\caption{The Dyck path for the juggling sequence in Figure~\ref{fig:recurout}}
\label{fig:Dyck}
\end{figure}

Careful analysis of the bijection shows that a juggling card sequence with $b$ balls and $n$ cards will produce a Dyck path from $(0,0)$ to $(2n,0)$ which has $n+1-b$ peaks.  This latter statistic on Dyck paths is counted by the Narayana numbers (see {\tt A001263} in \cite{oeis}).  Establishing the following theorem (a generating function proof of which will be given later in this section).

\begin{theorem}\label{thm:narayana}
The number of minimal crossing juggling card sequences with $b$ balls and $n$ cards is 
\[
f(n,b)=\frac1b\binom{n-1}{b-1}\binom{n}{b-1} = \frac1n{n\choose b}{n\choose b-1}, 
\]
the Narayana numbers.
\end{theorem}

\subsection{Non-crossing partitions}
An alternative way to establish Theorem~\ref{thm:narayana} is to note that the Narayana numbers are the number of ways to partition $[n]$ into $b$ disjoint nonempty sets which are non-crossing, i.e., so that there are no $a<b<c<d$ so that $a,c\in S_i$ and $b,d\in S_j$ (e.g., see \cite{simion}).  The sets $S_i$, formed by the locations of when the $i$-th ball is thrown, form such a non-crossing partition (i.e., if such $a<b<c<d$ exist then balls $i$ and $j$ intersect at least \emph{three} times, which is impossible).  One then checks that using the same construction as in Theorem~\ref{fact1a} that we can go from a non-crossing partition to one of the juggling card sequences we are counting establishing the bijection.

The important observation to make here, and which we will rely on moving forward, is that if we know the ordering of the balls at the left and right ends \emph{and} we know the order in which the balls are thrown, then we can uniquely determine the cards.

\subsection{Counting using generating functions}
We will now give another proof of Theorem~\ref{thm:narayana} which will employ the use of generating functions.  We focus on looking at the \emph{ball throwing patterns} $P = \langle b_1, b_2, \ldots, b_n\rangle$ which list the balls thrown at each step.  Given that the minimal crossing juggling card sequences will have each of the $b$ balls thrown we have that $P$ is a partition of $[n]$ into $b$ nonempty sets which are ordered by smallest element.

We will find it convenient to consider a shorthand notation $P^* = \langle d_1,d_2, \ldots, d_r\rangle$ for a pattern $P$ where each $d_k$ denotes a block of $d_k$'s of length at least one, and adjacent $d_k$'s are distinct (note that repeated $d_k$'s correspond to use of the card $C_1$).  Thus, if $P=\langle1,1,1,2,2,2,1,3,3,3,3,2,2,4\rangle$ then the reduced pattern is $P^* =\langle1,2,1,3,2,4\rangle$.  As noted before, in the patterns that we are interested in counting, each pair of balls cross exactly twice and so there cannot be an occurrence of $\langle\ldots,a,\ldots,b,\ldots,a,\ldots,b,\ldots\rangle$ in $P^*$.

\begin{proof}[Proof of Theorem~\ref{thm:narayana}]
We now define the following generating functions:
\begin{equation*}
\begin{array}{r@{\,=\,}l}
F_b(y)&\sum_{n \geq 1} f(b,n) y^n, \\[5pt]
F(x,y)&\sum_{b,n \geq 1} f(b,n) x^b y^n  = \sum_{b,n \geq 1} F_b(y) x^n.
\end{array}
\end{equation*}

For $b=1$, we have $f(1,n)=1$ for all $n$, since the only possible juggling card sequence consists of $n$ identical cards $C_1$. Thus,
\begin{equation*}
F_1(y) = y + y^2 + y^3 + \cdots = \frac{y}{1-y}.
\end{equation*}

Let us consider the only possible reduced pattern $P^*=\langle1,2,\overline1\rangle$ of ball throwing patterns for $b = 2$. The notation $\overline1$ indicates that this block of $1$'s may be empty. Thus,
\begin{equation*}
F_2(y) = \frac{y}{1-y} F_1(y) \frac{1}{1-y}=\frac{y^2}{(1-y)^3}
\end{equation*}
where the fraction $\frac{1}{1-y}$ allows for the possibility that the second block of $1$'s may be empty (i.e., this is $1+F_1(y)$).

For $b=3$, there are two possibilities for the reduced pattern $P^*$.  The first is that $P^* =\langle1,C,\overline1\rangle$ where  $C$ consists of $2$'s and $3$'s (and both must occur). The second is that $P^*=\langle1,2,1,3,\overline1\rangle$. Thus, we have
\begin{equation*}
F_3(y) = \frac{y}{1-y} F_2(y) \frac{1}{1-y} + \frac{y}{1-y} F_1(y) \frac{y}{1-y} F_1(y) \frac{1}{1-y} = \frac{y^3(y+1)}{(1-y)^5}.
\end{equation*}

Now consider the case for a general $b \geq 3$. Here, we can also partition the possibilities for $P^*$ into two cases. On one hand, we can have $P^* = \langle1,C\rangle$ where $C$ is a pattern using all $b-1$ of the balls $\{2,3, \ldots, b\}$.  The number of possible reduced patterns in this case is $\frac{y}{1-y} F_{b-1}(y)$.  On the other hand, there may be additional $1$'s which occur after the first block of $1$'s. In this case $P^*$ has the form $\langle1,C_1,C_2\rangle$ where $C_1$ uses $i > 0$ balls (not including $1$), and $C_2$ begins with a $1$ and uses $j > 0$ balls (including $1$).  Note this decomposition is the same that was given in Lemma~\ref{lem:recurrence}.  Since $C_1 \cup C_2 = [b]$ then $i + j = b$. In this case the number of possible patterns is given by the following expression:
\begin{equation*}
\sum_{\substack{0<i<b\\i+j=b}}\frac{y}{1-y} F_i(y) F_j(y)
\end{equation*}
Therefore we have,
\begin{equation*} 
F_b(y) = \frac{y}{1-y} F_{b-1}(y) + \sum_{\substack{0<i < b\\i+j=b}}\frac{y}{1-y}F_i(y)F_j(y)
\end{equation*}
Multiplying both sides by $x^b$ and summing over $b \geq 2$, we obtain
\begin{align*}
 F(x,y) - xF_1(y) &=\sum_{b \geq 2} F_b(y) x^b\\ 
&= \frac{y}{1-y} \sum_{b \geq 2} x^bF_{b-1}(y) + \frac{y}{1-y} \sum_{b\geq 2}\sum_{\substack{0<i<b,\\i+j=b}} x^iF_i(y) \,x^jF_j(y)\\ 
&= \frac{y}{1-y} ( x F(x,y) + \big(F(x,y)\big)^2 )
\end{align*}
In other words,
\begin{equation}
\label{eq3}
y \big(F(x,y)\big)^2 = (1-y-xy) F(x,y) -xy.
\end{equation}
Solving this for $F(x,y)$, we get
\begin{align*}
F(x,y)
&= \frac{1}{2y}\big(1-y-xy-\sqrt{(1-y-xy)^2 - 4xy^2}\big)\\
&= \frac{1}{2y}\big(1-y-xy-\sqrt{(1+y-xy)^2 - 4y}\big)\\
&= \frac{1}{2y}\bigg(1-y-xy-(1+y-xy) \sqrt{1-\frac{4y}{(1+y-xy)^2}}\bigg)\\
&= \frac{1}{2y}\bigg( 1-y-xy\\&\qquad\qquad-(1+y-xy)+(1+y-xy)\sum_{k \geq 1}\frac{(2k-2)!}{2^{2k-1} k! (k-1)!} \frac{(4y)^k}{(1+y-xy)^{2k}}\bigg)\\
&= \frac{1}{2y}\left( -2y+4y\sum_{k \geq 0}\frac{(2k)!}{2^{2k+1} (k+1)! k!} \frac{(4y)^k}{(1+y-xy)^{2k+1}}\right)\\
&= -1 +\sum_{k \geq 0}\frac{1}{k+1}\binom{2k}{k} y^k \sum_{j \geq 0} \binom{2k+j}{j} y^j (x-1)^j.
\end{align*}
Extracting the coefficient of $x^b y^n$, we obtain
\begin{equation*}
f(b,n) = \sum_{k \geq 0} \frac{1}{k+1} \binom{2k}{k} \binom{n+k}{n-k} \binom{n-k}{b} (-1)^{n-b-k}.
\end{equation*}
It remains to check that the right-hand side reduces to $\frac{1}{b} \binom{n}{b-1} \binom{n-1}{b-1}$. Rewriting the right hand side, we obtain 
\begin{equation*}
f(b,n)=\frac{1}{b}\binom{n-1}{b-1}\sum_{k \geq 0} \binom{n+k}{k+1} \binom{n-b}{k} (-1)^{n-b-k}.
\end{equation*}
Thus, our proof will be complete if we can show
\begin{equation*}
\sum_{k \geq 0} (-1)^{n-b-k} \binom{n+k}{k+1} \binom{n-b}{k} = \binom{n}{b-1}.
\end{equation*}
However, this follows at once by identifying the coefficients of $x^b$ in the expressions
\begin{equation*}
\frac{1}{(1-x)^n} (1-x)^{n-b} = (1-x)^{-b}.
\end{equation*}
\end{proof}

Knowing that
\begin{equation*}
F(x,y) = \sum_{b \geq 1} \sum _{n \geq b} \frac{1}{b} \binom{n}{b-1} \binom{n-1}{b-1} x^b y^n,
\end{equation*}
we can substitute into \eqref{eq3} and identify the coefficients of $x^b y^n$ to obtain the following interesting binomial coefficient identity.

\begin{corollary}
We have
\begin{equation*}
\sum_{\substack{1 \leq i \leq b-1\\1 \leq j \leq n-2}}\frac{1}{i (b-i)}\binom{j}{i-1} \binom{j-1}{i-1} \binom{n-1-j}{b-i-1} \binom{n-2-j}{b-i-1} = \frac{2}{b} \binom{n-1}{b-2} \binom{n-2}{b-1}.
\end{equation*}
\end{corollary}

\section{Juggling card sequences with $b(b-1)+2$ crossings}
In the preceding section we looked at minimal crossing juggling card sequences.  In this section we want to look at the ones which are \emph{almost} minimal, in the sense that we will increase the number of crossings to $b(b-1)+2$.  We will focus on the analysis of the ball throwing patterns.

Since each pair of balls cross at least twice and will always cross an even number of times, then it must be the case that there is a special pair of balls, call then $a$ and $b$ with $a<b$, which cross four times.  Therefore the ball throwing pattern contains the pattern $\langle\ldots,a,\ldots,b,\ldots,a,\ldots,b,\ldots\rangle$.  It is possible that there might be additional copies of the $a$'s and $b$'s so that this problem is not equivalent to counting the number of partitions with one crossing, for which if has been shown (see \cite{bergerson, bona}) that the number of partitions of $[n]$ into $b$ sets which have exactly one crossing is $\binom{n}{b-2}\binom{n-5}{b-3}$.  Nevertheless, we will see that the answers are similar and in this section we will establish the following.

\begin{theorem}\label{thm:plustwo}
The number of juggling card sequences $A$ with $b$ balls, using $n$ cards one of which is $C_b$, having $\pi_A=\text{id}$ and $Cr(A)=b(b-1)+2$ is
\[
g(b,n)={n\choose b+2}{n\choose b-2}.
\]
\end{theorem}

\subsection{Structural result}
To help establish Theorem~\ref{thm:plustwo} it will be useful to understand the structure of these ball throwing patterns. 

\begin{lemma}\label{lem:plustwo}
A ball throwing pattern, $P$, of length $n$ using $b$ balls with two additional crossings can be decomposed into four ball throwing patterns with no additional crossings, $P_0$, $P_1$, $P_2$, $P_3$ where $P_i$ has length $m_i\ge1$ using $c_i\ge1$ balls, $m_0+m_2+m_2+m_3=n$, $c_0+c_1+c_2+c_3=b+2$, \emph{and} a choice of the location of an entry, $i_1$, in $P_0$.
\end{lemma}
\begin{proof}
The crossings between $a$ and $b$ will happen in four of the cards for the juggling card sequence, and using the ball throwing pattern we can determine precisely where this will happen.  Namely, we know that since $a<b$ then $a$ must at some first point be thrown higher than $b$ which will occur at the last occurrence of $a$ before the first occurrence of $b$ (i.e., the last time we throw $a$ before we see $b$); suppose this happens at $i_1$.  Then the next crossing happens at the last occurrence of $b$ before the first occurrence of $a$ after $i_1$; suppose this happens at $i_2$.  Then the next crossing happens at the last occurrence of $a$ before the first occurrence of $b$ after $i_2$; suppose this happens at $i_3$.  Finally the last crossing happens at the last occurrence of $b$ before the first occurrence of $a$ after $i_3$; suppose this happens at $i_4$.  In particular we have the following (where some of the ``$\ldots$'' might be empty):
\[
\begin{array}{r@{~~~}c@{}c@{}c@{}c@{}c@{}c@{}c@{}c@{}c}
\text{Ball throwing pattern:}&\langle\ldots,&a&,\ldots,&b&,\ldots,&a&,\ldots,&b&,\ldots\rangle\\
\text{Location of crossings:}&&i_1&&i_2&&i_3&&i_4 
\end{array}
\]
Note that there might be additional occurrences of $a$ and $b$ in the ball throwing pattern, so far we have focused only on the location of the crossings.

We now split the ball throwing pattern into four subpatterns $P_i$ as follows:
\begin{itemize}
\item $P_1$ consists of the entries of $P$ between $i_1+1$ and $i_2$ (inclusive).
\item $P_2$ consists of the entries of $P$ between $i_2+1$ and $i_3$ (inclusive).
\item $P_3$ consists of the entries of $P$ between $i_3+1$ and $i_4$ (inclusive).
\item $P_0$ consists of the remaining entries of $P$, namely up to $i_1$ \emph{and} after $i_4+1$.
\end{itemize}

Note that no subpattern contains \emph{both} $a$ and $b$ (by construction), and therefore each one of these subpatterns (by proper relabeling, i.e., so that the first occurrences of the balls in order are $1,2,\ldots$) give ball throwing patterns with no additional crossings.  So we have decomposed the ball throwing pattern into four patterns with no additional crossings, by construction the sum of the lengths of the subpatterns is $n$.  We further have the following which gives information about the number of palls in the subpatterns.

\begin{claim}
No ball other than $a$ and $b$ occurs in two of the $P_i$.
\end{claim}

To see this suppose that a ball $c$ occurred both in $P_1$ and $P_2$.  Then it must be the case that our pattern $P$ contains $\langle\ldots,c\ldots,b,\ldots,c\rangle$.  But this is impossible, because between the two occurrences of $c$ in the pattern $c$ had to go above $b$ (one crossing) and then $b$ had to go above $c$ (a second crossing) and so there are no more available crossings for $b$ and $c$ to interact.  However we know that the ordering on both ends is the identity and so there must be another crossing at some point either before or after the $c$'s to put them in the correct order at both ends giving us a third crossing which is impossible (since other than the pair $a$ and $b$, each pair crosses exactly twice).  The same argument works for each other pair of intervals.

Therefore we can conclude that $a$ appears in $P_0$ and $P_2$, $b$ appears in $P_1$ and $P_3$ and each other ball appears in exactly one of the $P_i$.  Letting $c_i$ denote the number of balls in each $P_i$ we can conclude that $c_0+c_1+c_2+c_3=b+2$.  Finally we note that the decomposition for $P$ involved splitting the interval for $P_0$ at some point, for which there are $m_0$ places we could have chosen (i.e., $i_1$ is something from $1,2,\ldots,m_0$).

To finish the bijection we now show how to take four patterns $P_0,P_1,P_2,P_3$ with no additional crossings with lengths $m_0+m_1+m_2+m_3=n$, number of balls $c_0+c_1+c_2+c_3=b+2$, and a choice $1\le i_1\le m_0$ to form a pattern $P$ with two additional crossings.  We start by first labeling the balls so that they are all distinct among all the $P_i$ and no balls are yet labeled $a$ and $b$ and carry out the following three steps:
\begin{enumerate}
\item Whichever ball is thrown in position $i_1$ in $P_0$ we relabel that ball $a$ in all its occurrences in $P_0$.  Whichever ball is thrown in position $m_1$ in $P_1$ we relabel that ball $b$ in all its occurrences in $P_1$.  Whichever ball is thrown in position $m_2$ in $P_2$ we relabel that ball $a$ in all its occurrences in $P_2$.  Whichever ball is thrown in position $m_3$ in $P_3$ we relabel that ball $b$ in all its occurrences in $P_3$.  (Note that we now have $b$ different labels in use.)
\item Form a ball throwing pattern by concatenating, in order, the first $i_1$ entries from $P_0$, all of $P_1$, all of $P_2$, all of $P_3$, and the remaining $m-i_1$ entries from $P_0$.
\item Relabel the balls so that the first occurrences of the balls in order are $1,2,\ldots$.
\end{enumerate}

This produces a ball throwing pattern which has $b(b-1)+2$ crossings (i.e., since $a$ and $b$ will cross four times and no other pair of balls can have more than two crossings).  Further, applying the preceding decomposition argument we can precisely recover $P_0,P_1,P_2,P_3$ and our choice of $i_1$, establishing the bijection.
\end{proof}

\subsection{Using generating functions}
As in the preceding section, we can define a generating function for what we are trying to count,
\[
G(x,y)=\sum_{b\geq2,n \geq 4} g(b,n) x^b y^n.
\]
We are now ready to establish Theorem~\ref{thm:plustwo}

\begin{proof}[Proof of Theorem~\ref{thm:plustwo}]
From Lemma~\ref{lem:plustwo} we know that the ball throwing patterns we want to count can be decomposed into four ball throwing patterns with no crossings and where there is a choice of where to make a split on the first pattern.  Therefore we have
\begin{equation}\label{eq:g}
g(b,n)=\sum_{\substack{c_i,m_i \ge1\\ c_0+c_1+c_2+c_3=b+2\\m_0+m_1+m_2+m_3=n}}
m_0f(c_0,m_0)f(c_1,m_1)f(c_2,m_2)f(c_3,m_3).
\end{equation}
We recall the generating function for the ball throwing patterns with no crossings (i.e., for minimal crossing juggling sequences),
\[
F(x,y)=\sum_{b,n\ge 1}f(b,n)x^by^n=\frac{1-y-xy-\sqrt{(1-y-xy)^2 - 4xy^2}}{2y},
\]
and note that 
\[
y\frac{\partial}{\partial y}\big(F(x,y)\big) = \sum_{b,n\ge 1}nf(b,n)x^by^n.
\]
If we now multiply both sides of \eqref{eq:g} by $x^by^n$ and then sum we have the following
\begin{align*}
G(x,y)
&=\sum_{b\geq2,n \geq 4} g(b,n) x^b y^n\\
&=\sum_{b\geq2,n \geq 4}\bigg(\sum_{\substack{1\le c_i,m_i\\ c_0+c_1+c_2+c_3=b+2\\m_0+m_1+m_2+m_3=n}}
m_0f(c_0,m_0)f(c_1,m_1)f(c_2,m_2)f(c_3,m_3)\bigg) x^by^n\\
&=\frac1{x^2}\sum_{b\geq2,n \geq 4}\sum_{\substack{1\le c_i,m_i\\ c_0+c_1+c_2+c_3=b+2\\m_0+m_1+m_2+m_3=n}}
\begin{array}{l}
\big(m_0f(c_0,m_0)x^{c_0}y^{m_0}\times f(c_1,m_1)x^{c_1}y^{m_1}\\\qquad\times f(c_2,m_2)x^{c_2}y^{m_2}\times f(c_3,m_3)x^{c_3}y^{m_3}\big)\end{array}\\
&=\frac1{x^2}\bigg(y\frac{\partial}{\partial y}\big(F(x,y)\big)\bigg)\times F(x,y)\times F(x,y)\times F(x,y)\\
&=y\frac{\partial}{\partial y}\bigg( \frac{\big(F(x,y)\big)^4}{4x^2}\bigg).
\end{align*}

Taking the known expression for $F(x,y)$ and letting $z=1-y-xy$ we have
\[
\frac{\big(F(x,y)\big)^4}{4x^2}=
\frac{8z^4-32xy^2z^2+16x^2y^4-(8z^3-16xy^2z)\sqrt{z^2-4xy^2}}{64x^2y^4},
\]
Further we have
\begin{align*}
\sqrt{z^2-4xy^2}
&= z\sqrt{1-\frac{4xy^2}{z^2}}\\
&=z-z\sum_{k\ge1}\frac{(2k-2)!}{2^{2k-1}k!(k-1)!}\frac{(4xy^2)^k}{z^{2k}}\\
&=z-\frac{2xy^2}z-z\sum_{k\ge 2}\frac{(2k-2)!}{2^{2k-1}k!(k-1)!}\frac{(4xy^2)^k}{z^{2k}}\\
&=z-\frac{2xy^2}{z}-2\sum_{k\ge0}\frac{(2k+2)!}{(k+2)!(k+1)!}\frac{x^{k+2}y^{2k+4}}{z^{2k+3}}.
\end{align*}
Substituting this in and simplifying we have
\begin{align*}
\frac{\big(F(x,y)\big)^4}{4x^2}&=\frac14(z^2-2xy^2)\sum_{k\ge 0}\frac{(2k+2)!}{(k+2)!(k+1)!}\frac{x^ky^{2k}}{z^{2k+2}}\\
&=\frac14\sum_{k\ge 0}\frac{(2k+2)!}{(k+2)!(k+1)!}\frac{x^ky^{2k}}{z^{2k}}-\frac12\sum_{k\ge 0}\frac{(2k+2)!}{(k+2)!(k+1)!}\frac{x^{k+1}y^{2k+2}}{z^{2k+2}}\\
&=\frac14+\frac12\sum_{k\ge2}\frac{(2k)!(k-1)}{k!(k+2)!}\frac{x^ky^{2k}}{z^{2k}},
\end{align*}
where in going to the last line we pull off the first term on the first summand and shift the second summand and then combine noting we can drop the $k=1$ case.  We also have
\[
\frac1{z^{2k}}=\frac1{\big(1-y(x+1)\big)^{2k}}=
\sum_{j\ge 0}{2k-1+j\choose j}y^j(x+1)^j.
\]
Substituting this we now have
\begin{align*}
\frac{\big(F(x,y)\big)^4}{4x^2}&=\frac14+\frac12\sum_{\substack{j\ge0\\k\ge2}}\frac{(2k)!(k-1)}{k!(k+2)!}{2k-1+j\choose j}x^k(x+1)^jy^{2k+j}.
\end{align*}
Finally, we can recover $G(x,y)$ since what remains is to take the derivative with respect to $y$ and then multiply by $y$, which is equivalent to bringing down the power of $y$.  After simplifying, we can conclude
\begin{align*}
G(x,y)&=\frac12\sum_{\substack{j\ge0\\k\ge2}}\frac{(2k)!(k-1)(2k+j)}{k!(k+2)!}{2k-1+j\choose j}x^k(x+1)^jy^{2k+j}\\
&=\sum_{\substack{j\ge0\\k\ge2}}{2k+j\choose k+2,k-2,j}x^k(x+1)^jy^{2k+j}\\
&=\sum_{\substack{n\ge 4\\k\ge2}}{n\choose k+2,k-2,n-2k}x^k(x+1)^{n-2k}y^n,
\end{align*}
where ${a\choose b,c,d}$ is the multinomial coefficient ${a!\over b!c!d!}$ and in going to the last line we make the substitution $j\to n-2k$.

We can now get the coefficient of $x^by^n$, which is done by using the binomial theorem and summing over possible $k$.  In particular we can conclude
\begin{align*}
g(b,n)&=\sum_k{n\choose k+2,k-2,n-2k}{n-2k\choose b-k}\\
&=\sum_k{n\choose k+2,k-2,b-k,n-b-k}.
\end{align*}
By the special case $a=2$ of Proposition~\ref{prop:binomialidentity} (given below) this is equal to ${n\choose b+2}{n\choose b-2}$, finishing the proof.
\end{proof}

\begin{proposition}\label{prop:binomialidentity}
$\displaystyle \sum_k{n\choose k+a,k-a,b-k,n-b-k} = {n\choose b+a}{n\choose b-a}$.
\end{proposition}
\begin{proof}
We count the number of ways to select two sets $A$ and $B$ from $n$ elements, with $|A|=b+a$ and $|B|=b-a$.  This is clearly equal to the right hand side, so it remains to show how the left hand side equals this as well.

We begin by noting that we can rewrite the multinomial coefficient as a product of binomial coefficients in the following way,
\[
\sum_k{n\choose k+a,k-a,b-k,n-b-k}=\sum_k{n\choose b+k}{b+k\choose 2k}{2k\choose k+a}.
\]
We now choose our sets in the following way:  First we pick $b+k$ elements which will correspond to $A\cup B$, then among those $b+k$ elements we choose the $2k$ elements which will belong to precisely one of the sets, finally among the $2k$ elements which will belong to exactly one set we choose $k+a$ of them for $A$ and the remaining $k-a$ go to $B$.  Summing over all possibilities for $k$ now gives the desired count.
\end{proof}

\subsection{Higher crossing numbers}
The next natural step in our problem is to ask for the enumeration of sequences $A$ with larger values of $Cr(A)$. One approach to this problem would be to further simplify the types of juggling card sequences we are counting. Let us call a juggling card sequence $A$ \emph{primitive} if it does not use the ``trivial'' card $C_1$, i.e., the card which generates the identity permutation. Such a card does not contribute to the number of crossings $Cr(A)$ of $A$.  We note that counting these primitive juggling card sequences is equivalent to counting the reduced ball throwing patterns which do not end in $1$.

Let us denote by $P_d(n,b)$ the number of primitive juggling card sequences $A$ with $n$ cards using the card $C_b$ with  $\pi(A) = \text{id}$ and $Cr(A) = b(b-1)+d$, and let $Q_d(n,b)$ denote the number of such sequences which are not necessarily primitive. Since crossings occur in pairs, $d$ must be even. Then 
\[
   Q_d(n,b) = \sum_{k=1}^n \binom{n}{k} P_d(k,b).
\]
The hope would be that $P_d(n,b)$ could be simpler in some sense than $Q_d(n,b)$ and would therefore  be easier to recognize. It turns out that if we write $n = b+t$ then it is not hard to show that
\[
    P_0(n,b) = \frac{1}{t+1} \binom{b-2}{t} \binom{b+t}{t}
\]
and
\[
    P_2(n,b) =\binom{b+t}{2t} \binom{2t}{t-2}.
\]
In Table~\ref{tab:P4} we give data (in factored form) for $P_4(n,b)$ for small values of $n$ and $b$.

\begin{table}[t]
\centering
\caption{Data for $P_4(n,b)$}\label{tab:P4}

\bigskip

\begin{tabular} {|c||c|c|c|c|c|c|c|} 
\hline
$P_4(n,b)$ & $b{=}2$ & $b{=}3$ & $b{=}4$ & $b{=}5$ & $b{=}6$ & $b{=}7$ & $b{=}8$\\ \hline \hline
$n{=}6$ & 1 & 3 &  & & & & \\ \hline
$n{=}7$ & & $2 {\cdot} 7$ & $3 {\cdot} 7$ & & & &\\ \hline
$n{=}8$ & & $2^2 {\cdot} 3$ & $2^4 {\cdot} 7$ & $2^2 {\cdot} 3 {\cdot} 7$ & & &\\ \hline
$n{=}9$ & & & $2^2 {\cdot} 3 ^2 {\cdot} 5$ & $2^2 {\cdot} 3 {\cdot} 7^2$ & $2^3 {\cdot} 3^2 {\cdot} 7$ & &\\ \hline
$n{=}10$ & & & $2 {\cdot} 3^2 {\cdot} 5$ & $2^5 {\cdot} 3^2 {\cdot} 5$ & $2 {\cdot} 3 {\cdot} 5 {\cdot} 7 {\cdot} 11$ & $2 {\cdot} 3^2 {\cdot} 5 {\cdot} 7$ & \\ \hline
$n{=}11$ & & & & $3^3 {\cdot} 5 {\cdot} 11$ & $2^4 {\cdot} 3^3 {\cdot} 5 {\cdot} 11$ & $2^5 {\cdot} 3 {\cdot} 7 {\cdot} 11$ & $2 {\cdot} 3^2 {\cdot} 7 {\cdot} 11$\\ \hline
\end{tabular}

\bigskip

\begin{tabular} {|c||c|c|c|c|c|} 
\hline
$P_4(n,b)$ & $b{=}5$ & $b{=}6$ & $b{=}7$ & $b{=}8$ & $b{=}9$ \\ \hline \hline
$n{=}12$ & $2 {\cdot} 5^2 {\cdot} 11$ & $2 {\cdot} 3^2 {\cdot} 5 {\cdot} 11 {\cdot} 13$ & $2^2 {\cdot} 3^2 {\cdot} 5 {\cdot} 11 {\cdot} 17$
& $2^3 {\cdot} 3 {\cdot} 7 {\cdot} 11^2$ &$2^2 {\cdot} 3^2 {\cdot} 7 {\cdot} 11$  \\ \hline
$n{=}13$ & & $2 {\cdot} 5 {\cdot} 7 {\cdot} 11 {\cdot} 13$ & $2^2 {\cdot} 3^3 {\cdot} 5 {\cdot} 11 {\cdot} 13$ 
& $2^2 {\cdot} 3^2 {\cdot} 11 {\cdot} 13 {\cdot} 23$ &  $2^2 {\cdot} 3 {\cdot} 11 {\cdot} 13 {\cdot} 29$  \\ \hline
$n{=}14$ & & $3 {\cdot} 7 {\cdot} 11 {\cdot} 13$ & $5 {\cdot} 7 {\cdot} 11 {\cdot} 13 {\cdot} 19$& $2^3 {\cdot} 3^2 {\cdot} 5 {\cdot} 7 {\cdot} 11 {\cdot} 13$ &
$2^3 {\cdot} 3^2 {\cdot} 5 {\cdot} 7 {\cdot} 11 {\cdot} 13$ \\ \hline
\end{tabular}
\end{table}

The fact that there are many small factors suggest that $P_4(n,b)$ could be made up of binomial coefficients in some way. However, the presence of occasional ``large'' factors makes it difficult to guess what the expressions might actually be (for example, $P_4(14,10) = 3 {\cdot} 7 {\cdot} 11 {\cdot} 13 {\cdot} 37$). Nevertheless, computations suggested that $P_4(n,b)$ is given by the following expression:
\[
   P_4(n,b)=\frac{(bn-b-8)}{2(b+4)}\binom{n}{b+3}\binom{n}{b-2}.
\]
This has been confirmed by one of the authors (Cummings), but we do not give a proof of this result here.  We don't even have a guess as to what the expressions are for $P_{2k}$ when $k \geq 3$!

\section{Final arrangements consisting of a single cycle}
Suppose that we draw cards at random from the set $\{C_1, C_2, \ldots, C_b\}$ with replacement to form a juggling card sequence $A$. We can then ask for the probability that $\pi_A$ has some particular property.  For example, what is the probability that it is equal to some given permutation, such as the identity, or that the permutation consists of a single cycle. The first question can be answered using Theorem~\ref{fact1a}. The answer for the second question is especially nice.  We state the result as follows.

\begin{theorem}\label{fact4}
The probability that a random sequence $A$ of $n$ cards taken from the set of juggling sequence cards $\{C_1, C_2, \ldots, C_b\}$ has $\pi_A$ consisting of a single cycle is $1/b$.  In particular, this is independent of $n$.
\end{theorem}

The following proof is due to Richard Stong \cite{stong}.  We start with the following two basic lemmas.

\begin{lemma}\label{obs1}
The probability that a random permutation $\sigma$ of $[b]$ has $L(\sigma) \geq k$ is $1/k!$ for $1 \leq k \leq b$.
\end{lemma}

\begin{proof}
Select a $k$-element subset $\{a_0 > a_1>\cdots> a_{k-1}\}$ from $[b]$. Define the permutation $\rho$ by first setting $\rho(b-i) = a_i$ for $0 \leq i \leq k-1$. There are exactly $(b-k)!$ ways to complete $\rho$ so that it is a permutation of $[b]$. Clearly, $L(\sigma) \geq k$ and there are $\binom{b}{k} (b-k) = b!/k!$ choices for $\rho$ (and any $\rho$ with $L(\rho) = k$ must be formed this way). Thus, the probability that a random $\rho$ has $L(\rho) \geq k$ is $1/k!$ as claimed.
\end{proof}

We note here that the number of permutations of $[b]$ that consist of a single cycle is $(b-1)!$.

\begin{lemma}\label{obs2}
The probability that a random permutation $\sigma$ of $[b]$ which consists of a single cycle has $L(\sigma) \geq k$ is $1/k!$ for $1 \leq k \leq b-1$.
\end{lemma}

\begin{proof}
The proof is similar to that of Lemma~\ref{obs1}. In this case we choose $k$ elements $\{ a_0 > a_1> \cdots > a_{k-1}\}$ from $[b-1]$ and map $\rho(b-i)$ to $a_i$ for $0 \leq i \leq k-1$ as before. The reason that we don't allow $a_0 = b$ is that if $\rho(b) = a_0 = b$ then $\rho$ would have a fixed point and so, could not be a single cycle.  Now the question is how to complete the definition of $\rho$ so that it becomes a single cycle. This is actually quite easy. We have the beginning of $b-k$ chains, namely, $b \rightarrow a_0$, $b-1 \rightarrow a_1$, \ldots, $b-k+1 \rightarrow a_{k-1}$, together with the remaining single points not included in the points listed so far. It is just a matter of piecing these fragments together to form a single cycle. The fact that some of the $a_i$ might be equal to some of the $b-j$ causes no problem. It is easy to see that there are just $(b-k-1)!$ ways to complete the definition of $\rho$ so that it becomes a single cycle with $L(\rho) \geq k$, and furthermore all such $\rho$ can be constructed this way.  Since $\binom{b-1}{k} (b-k-1)! = (b-1)!/k!$, and there are $(b-1)!$  permutations of $[b]$ that are cycles of length $b$, this completes the proof of Lemma~\ref{obs2}.
\end{proof}

\begin{proof}[{Proof of Theorem~\ref{fact4}}]
Partition the set of $b!$ permutations of $[b]$ into $b$ disjoint classes $X_k$, for $1 \leq k \leq b$.  Namely, $\sigma \in X_k$ if and only if $L(\sigma) = k$. By Lemma~\ref{obs1}, $|X_k| = b! \big( \frac{1}{k!} - \frac{1}{(k+1)!} \big)$ for $1 \leq k \leq b-1$, while $|X_b| = 1$. Similarly, we can partition the set of $(b-1)!$ permutations which are $b$-cycles into disjoint sets $Y_k$, for $1 \leq k \leq b-1$, where $\sigma \in Y_k$ if and only if $L(\sigma) = k$.  By Lemma~\ref{obs2}, $|Y_k| = (b-1)! \big( \frac{1}{k!} - \frac{1}{(k+1)!} \big)$ for $1 \leq k \leq b-2$, while $|Y_{b-1}| = 1$. Note that $L(\sigma) \geq b-1$ if and only if $\sigma \in X_{b-1} \cup X_b$.

Now by Theorem~\ref{fact1a}, each $\sigma \in X_k$ accounts for exactly $\sum_{k=b-L(\sigma)}^b \big\{\!{n\atop k}\!\big\}$ different card sequences $A$ with $\pi_A = \sigma$, and the same is true for each $\sigma \in Y_k$, where $1 \leq k \leq b-2$. Furthermore, $|X_k| = b|Y_k|$ for these $k$. In addition, each $\sigma \in X_{b-1} \cup X_b$ and each $\sigma \in Y_{b-1}$ accounts for exactly $\sum_{k=1}^b \big\{\!{n\atop k}\!\big\}$ different card sequences $A$ with $\pi_A = \sigma$. Thus, since  $|X_{b-1} \cup X_b| = \frac{b!}{(b-1)!} = b = b|Y_{b-1}|$ then it follows that  the number of card sequences accounted for by all $\sigma$ (which is $b^n$) is exactly $b$ times the number accounted for by the $\sigma$ which are $b$-cycles. In other words, the probability that a random sequence of $n$ cards generates a permutation which is a $b$-cycle is just $1/b$, independent of $n$.
\end{proof}

It turns out that the analog of Theorem~\ref{fact4} holds for cards where $m$ balls are thrown.

\begin{theorem}\label{fact5}
The probability that a random sequence $A$ of length $n$ using cards where $m$ balls are thrown at a time has $\pi_A$ equal to a $b$-cycle is $1/b$. In particular, this is independent of $n$.
\end{theorem}

The proof follows the same lines as the proof of Theorem~\ref{fact4} and will be omitted.  The basic point is that in this case each $\sigma$ with $L(\sigma) = k$ accounts for exactly $\sum_{k=b-L(\sigma)}^b \big\{\!{n\atop k}\!\big\}_{\!m}$ sequences of $m$-cards with $\pi_A = \sigma$.  Note that it is not obvious that Theorem~\ref{fact5} even holds for $n=1$.

The surprising thing is that these results apply for all $n$ and is not tied to a limiting process.  Indeed, in the limit this is a special case of a much more general group theoretic principle that we prove now.

\begin{theorem}\label{fact6}
Let $G$ be a group, let $\mathcal{S} = \{g_1,\dots, g_k\}$ be a generating set of $G$, and let $\mathcal{P} = \{p_1,\dots,p_k\}$ be a corresponding set of non-zero probabilities summing to $1$.  Consider the Markov chain on $G$ where at each stage the current element is multiplied by a random $g \in \mathcal{S}$ chosen with probability given by $\mathcal{P}$.  Then the stationary distribution of this process is the uniform distribution, independent of the group structure or $\mathcal{P}$.
\end{theorem}

\begin{proof}
For simplicity we will assume that our walk begins at the identity element.  Consider the formal sum $\mathcal{D} = \sum_{g_i\in S} p_i g_i$.  The probability distribution of the random walk after $n$ steps is then given by the formal sum $\mathcal{D}^n$.  Let $\mathcal{F} = \sum_{g\in G}q_gg$ be the stationary distribution of this Markov chain.  Then we have that $\mathcal{F}$ acts as a fixed point, i.e., $\mathcal{D}\mathcal{F} = \mathcal{F}$.

Let $h$ be a group element whose probability $q_h$ in the stationary distribution is maximum, i.e., $q_h \geq q_g$ for all $g \in G$.  Applying this after equating the $h$ coefficients on each side of $\mathcal{D}\mathcal{F} = \mathcal{F}$ gives
\[
    q_h = \sum_{i=1}^k p_i q_{g_i^{-1} h} \leq \sum_{i=1}^k p_i q_{h} = q_h,
\]
which can only hold if each $q_{g_i^{-1} h} = q_h$.  Now, for each $i$, apply this same argument by choosing $g_i^{-1}h$ as the maximum element instead of $h$.  Since $\{ g_i^{-1} : i \in [k]\}$ is also a generating set of $G$, by continuing in this way we see that $q_g = q_h$ for all $g \in G$, completing the proof.
\end{proof}

Thus in the case of $S_n$, the probability of having $\ell$ distinct cycles after choosing $n$ random juggling cards tends to $\big[{b\atop \ell}\big]/b!$ as $n$ tends to infinity, where $\big[{b\atop \ell}\big]$ indicates the Stirling number of the first kind, i.e., the number of ways to decompose $\{1,\ldots,b\}$ into $\ell$ disjoint cycles.  Indeed, we note without proof that it converges to this quite rapidly.  By following the lines of the proof of Theorem~\ref{fact4}, only in the ``end cases'' where $L(\sigma)$ is within $\ell$ of $b$ does the proportion not equal precisely $\big[{b\atop \ell}\big]/b!$.

\subsection*{Acknowledgment}
Part of this research was conducted while Steve Butler was a visitor at the Institute for Mathematics and its Applications.  In addition, Steve Butler was partially supported by an NSA Young Investigator Grant.


\begin{thebibliography}{99}
\bibitem{bergerson} M. Bergerson, A. Miller, A. Pliml, V. Reiner, P. Shearer, D. Stanton, and N. Switala, \emph{Note on $1$-crossing partitions} (available at \url{www.math.umn.edu/~reiner/Papers/onecrossings.pdf}).
\bibitem{boson_1} P. Blasiak, K. A. Penson, and A. I. Solomon, The Boson Normal Ordering Problem and Generalized Bell Numbers, \emph{Annals of Combinatorics} \textbf{7}, (2003), 127--139.
\bibitem{bona} M. B\'ona, Partitions with $k$ crossings, \emph{The Ramanujan Journal} \textbf{3} (1999), 215--220.
\bibitem{BG} J. Buhler, and R. Graham, Juggling, passing and posets, \emph{Mathematics for Students and Amateurs}, Mathematical Association of America, (2004), 99--116.
\bibitem{BEGW} J. Buhler, D. Eisenbud, R. Graham, and C. Wright, Juggling drops and descents, \emph{Amer. Math. Monthly} \textbf{101} (1994), 507--519.
\bibitem{cps} P. Cameron, T. Prellberg, and D. Stark, Asymptotic enumeration of 2-covers and line graphs, \emph{Discrete Math.} \textbf{310} (2010), 230--240.
\bibitem{hell} P. Codara, O. D'Antona, and P. Hell, A simple combinatorial interpretation of certain generalized Bell and Stirling numbers,  arXiv:1303.1400v1 [cs.DM] 7 Aug 2013.
\bibitem{EGH} J. Engbers, D. Galvin, and J. Hilyard, Combinatorially interpreting generalized Stirling numbers, arXiv:1308.2666  (Nov. 2013).
\bibitem{ER} R. Ehrenborg, and M. Readdy, Juggling and applications to $q$-analogues, \emph{Discrete Mathematics}  \textbf{157} (1996), 107--125.
\bibitem{gkp} R. L. Graham, D. E. Knuth, and O. Patashnik, Concrete Mathematics: A Foundation for Computer Science, Addison-Wesley, Reading, MA, 2nd ed., 1994.
\bibitem{labele} G. Labele, Counting enriched multigraphs according to the number of their edges (or arcs), \emph{Discrete Math.} \textbf{217} (2000), 237--248.
\bibitem{Man} T. Mansour, Combinatorics of Set Partitions, \emph{CRC Press}, 2013.
\bibitem{boson_2} M. A. Mendez, P. Blasiak, and K. A. Penson, \emph{Combinatorial approach to generalized Bell and Stirling numbers and boson normal ordering problem}, \emph{Jour. of Math. Physics} \textbf{46} (2005), 083511-1--8.
\bibitem{amo} A. M. Odlyzko, Asymptotic Enumeration Methods, in: R. L. Graham, M. Gr\"otschel and L. Lov\'asz (Eds.), \emph{Handbook of Combinatorics, vol. 2}, North-Holland, Amsterdam, 1995, 1063--1229.
\bibitem{simion} R. Simion, Noncrossing partitions, \emph{Discrete Mathematics} \textbf{217} (2000), 367--409.
\bibitem{oeis}  N. Sloane, On-line Encyclopedia of Integer Sequences, (\url{oeis.org}).
\bibitem{stong}  R. Stong, (personal communication).
\end{thebibliography}
\end{document}